\documentclass[11pt]{article}
\usepackage[T1]{fontenc}
\usepackage{lmodern}
\usepackage[margin=1in]{geometry}
\usepackage{amsmath,amssymb,amsthm,mathtools}
\usepackage{microtype,enumitem,booktabs,longtable,array,placeins}
\usepackage{tikz}
\usetikzlibrary{positioning,arrows.meta,calc}
\usepackage[hidelinks,bookmarksnumbered]{hyperref}
\hypersetup{pdftitle={Watkins's conjecture holds for all infinite groups},
 pdfauthor={Alex J Sutherland},pdfsubject={Watkins's graphical regular representation conjecture, Cayley index, and hierarchy reconstruction},
 pdfkeywords={Watkins conjecture, Cayley index, graphical regular representation, infinite group, hierarchy reconstruction}}
\usepackage{fancyhdr}
\newtheorem{maintheorem}{Theorem}

\newtheorem{theorem}{Theorem}[section]
\newtheorem{lemma}[theorem]{Lemma}
\newtheorem{proposition}[theorem]{Proposition}
\newtheorem{corollary}[theorem]{Corollary}
\theoremstyle{definition}\newtheorem{definition}[theorem]{Definition}
\theoremstyle{remark}\newtheorem{remark}[theorem]{Remark}
\newcommand{\Cay}{\operatorname{Cay}}
\newcommand{\Aut}{\operatorname{Aut}}
\newcommand{\Sym}{\operatorname{Sym}}
\newcommand{\CN}{\operatorname{CN}}
\DeclareMathOperator{\maxdeg}{maxdeg}
\DeclareMathOperator{\Diff}{Diff}

\newcommand{\F}{\mathbb F}
\newcommand{\N}{\mathbb N}

\newcommand{\eps}{\varepsilon}
\newcommand{\invp}[1]{[#1]_{\pm}}
\newcommand{\cl}[1]{\langle #1\rangle}

\title{Watkins's conjecture holds for all infinite groups}
\author{Alex J Sutherland\thanks{\raggedright Oregon State University, Corvallis, OR 97331, USA.
Email: \href{mailto:suthalex@oregonstate.edu}{\texttt{suthalex@oregonstate.edu}}.
ORCID: \mbox{\href{https://orcid.org/0000-0002-1364-5663}{0000-0002-1364-5663}}.}}
\date{}
\newcommand{\diam}{\operatorname{diam}}

\newcommand{\cf}{\operatorname{cf}}
\begin{document}
\maketitle
\begin{abstract}
We prove that at every infinite cardinality, every group which is neither abelian of exponent greater than two nor generalized dicyclic admits a graphical regular representation, settling the infinite-group part of Watkins's conjecture. We also determine the Cayley index of every infinite group: it is $1$, $2$, or $8$, according to its algebraic type, and in every case the index is attained by a connected Cayley graph. For every infinite group $G$ of cardinality $\kappa$, we construct $2^\kappa$ pairwise nonisomorphic Cayley graphs with exactly the unavoidable inverse-pair symmetries, diameter two, and $\kappa$ common neighbors at every distinct pair. The principal tool recovers a continuous ordinal hierarchy from an alternating adjacency baseline with bounded-degree errors: robust finite patterns identify the initial classes, successive twin quotients recover the layers, and their finite exception packets determine the translation action. The reconstruction applies without a group action and is stable under additional layerwise bounded-degree edits. Further results give closed Cantor-cube families with prescribed finite data in the regular cases and optimal three-valued shortest-path metrics.
\end{abstract}
\noindent\small\textbf{2020 Mathematics Subject Classification.} Primary 05C25; Secondary 05C63, 20B07, 20B27.\par
\noindent\small\textbf{Keywords.} Cayley graph; graphical regular representation; infinite group; subgroup filtration; intrinsic reconstruction; invariant metric.\normalsize

\section{Watkins's conjecture and the Cayley index}
A \emph{graphical regular representation} (GRR) of a group is a simple undirected
Cayley graph whose full automorphism group consists of the prescribed translations.
A \emph{connection set} in $G$ means an inverse-closed subset of $G\setminus\{1\}$;
it is not required to generate $G$ unless this is stated. Its Cayley graph is
connected exactly when it generates $G$.
We settle the existence question for infinite groups, without any restriction on
cardinality.

\begin{maintheorem}[Watkins's conjecture for infinite groups]\label{thm:watkins}
Every infinite group that is neither abelian of exponent greater than two nor
generalized dicyclic admits a graphical regular representation.
\end{maintheorem}

A group is \emph{Boolean} if every element has order at most two. A group is \emph{abelian of exponent greater than two} if it is abelian and has an element of order greater than two; equivalently, it is abelian and not Boolean, the trivial group counting as Boolean. A
\emph{generalized dicyclic} group is nonabelian, has an abelian subgroup of
index two, and has an element of order exactly four outside that subgroup which
acts on it by inversion. These two families are unavoidable: each such group has a nontrivial automorphism sending every element to itself or its inverse, and such an automorphism fixes every connection set (see the remark after Theorem~\ref{thm:regularExist}). The \emph{regular cases} are the Boolean groups and the
nonabelian groups that are not generalized dicyclic. The \emph{Cayley index}
$i(G)$ is the infimum, over symmetric generating sets, of the index of the
prescribed regular subgroup in the full graph automorphism group.

\begin{maintheorem}[The Cayley index of every infinite group]\label{thm:cayleyindex}
For every infinite group $G$,
\[
i(G)=
\begin{cases}
1, & G\text{ is in a regular case};\\[2pt]
2, & \begin{array}[t]{@{}l@{}}
G\text{ is non-Boolean abelian, or is generalized dicyclic}\\
\text{and not isomorphic to }Q_8\times B\text{ for any Boolean group }B;
\end{array}\\[2pt]
8, & G\cong Q_8\times B\text{ for an infinite Boolean group }B.
\end{cases}
\]
In every case the index is attained by a connected Cayley graph.
\end{maintheorem}

Here $Q_8$ is the quaternion group of order eight. To state the stronger
realization result, write $\invp{g}=\{g,g^{-1}\}$ and let
\[
\Xi(G)=\{F\in\Sym(G):\invp{F(x)F(y)^{-1}}=\invp{xy^{-1}}
                         \text{ for all }x,y\in G\}.
\]
This is the group of permutations preserving every inverse-pair difference
relation, and is contained in the automorphism group of every undirected
Cayley graph on $G$. A \emph{realization} of $\Xi(G)$ is a Cayley graph
$\Gamma=\Cay(G,S)$ with $\Aut(\Gamma)=\Xi(G)$; we also call $\Gamma$ a
\emph{realizing graph} and $S$ a \emph{realizing connection set}.

For a graph $\Gamma$, let $\CN_\Gamma(x,y)$ denote the number of common neighbors of $x$ and $y$.

\begin{maintheorem}[Exact realization, abundance, and optimal diameter]\label{thm:all}
For every infinite group $G$ of cardinality $\kappa$, there are $2^\kappa$
pairwise nonisomorphic Cayley graphs $\Gamma$ on the original set $G$ such that
\[
\Aut(\Gamma)=\Xi(G),\qquad
\CN_\Gamma(x,y)=\kappa\quad\text{for all }x\ne y,\qquad \diam\Gamma=2.
\]
In the regular cases these graphs are GRRs. The number of uncolored
isomorphism types and the diameter are optimal.
\end{maintheorem}

The inverse-pair classification of Leemann and de la Salle (Theorem~\ref{thm:inversepairindex} below) shows that the existence assertion in Theorem~\ref{thm:all} implies Theorems~\ref{thm:watkins} and~\ref{thm:cayleyindex}.

The main tool is the group-independent reconstruction theorem
(Theorem~\ref{thm:recovery}), which recovers an ordinal hierarchy and its finite
exception packets directly from adjacency. In the regular-case construction, those
packets force the prescribed translation action.

Beyond the main theorems, the construction gives cofinality-scale connectivity,
domination, and deletion properties, together with closed Cantor-cube families
inside finite cylinders (Section~\ref{sec:consequences}). The reconstruction is
stable under additional layerwise bounded-degree edits
(Corollary~\ref{cor:edits}). Diameter two also yields the three-valued path
metrics recorded in Corollary~\ref{cor:metric}.

\subsection{Idea of the proof}\label{sec:proofidea}

Suppose $G$ is in a regular case. We choose a countably infinite subgroup $H$ with the same property, as justified in Lemmas~\ref{lem:witness} and~\ref{lem:core}. To extend a GRR from $H$ to $G$, we make the $H$-cosets recognizable by finite induced patterns that cannot cross between them, then synchronize the translations on those cosets using finite sets of exceptional elements. Lemma~\ref{lem:BP}--Theorem~\ref{thm:countable} construct the required core; we first explain the extension.

Starting from $H$, fix an enumeration of $G$ and adjoin to $G_\alpha$ the least-indexed element $a_\alpha$ outside it; this gives a continuous well-ordered chain $(G_\alpha)$ of subgroups exhausting $G$, with no last layer. On each layer $G_{\alpha+1}\setminus G_\alpha$, membership in the connection set is a constant bit $\eps_\alpha$ except on a finite inverse-closed packet $W_\alpha$. Packet sizes are bounded by a fixed $m$, and the bits alternate at successor indices. Thus every vertex has at most $m$ deviations from the baseline in each layer, even though its total number of deviations may be infinite.

In the core, the edges lying in finite induced $m$-robust patterns form a connected spanning graph. Robustness means that every cut has both a vertex with more than $m$ crossing neighbors and a vertex with more than $m$ crossing nonneighbors. Such a pattern cannot meet two $H$-cosets: at its highest occupied layer, a cut separates it into parts whose crossing adjacency is constant apart from at most $m$ errors per vertex. For baseline zero this bounds every crossing degree by $m$; for baseline one it bounds every crossing nondegree by $m$. Both contradict robustness. Consequently, the connected components of the edges lying in robust patterns are exactly the $H$-cosets.

Once these cosets are recognized, the baseline between two of them is read from the graph: a vertex has either at most $m$ neighbors or at most $m$ nonneighbors in the other coset. Since the cosets are infinite, the alternatives are disjoint. In the recovered baseline graph, the twin classes are exactly the next cosets in the hierarchy. Alternation supplies a vertex distinguishing any two classes that should not yet merge. Repeating the twin quotient, and taking unions at limit stages, recovers every $G_\alpha$-coset partition in its ordinal order, together with all baseline bits and error relations (Theorem~\ref{thm:recovery} and Corollary~\ref{cor:cayleyrecovery}). The exception packets can destroy the module structure, but this reconstruction needs neither modules nor a group operation. Figure~\ref{fig:recovery} illustrates these successor steps in the recovered baseline.

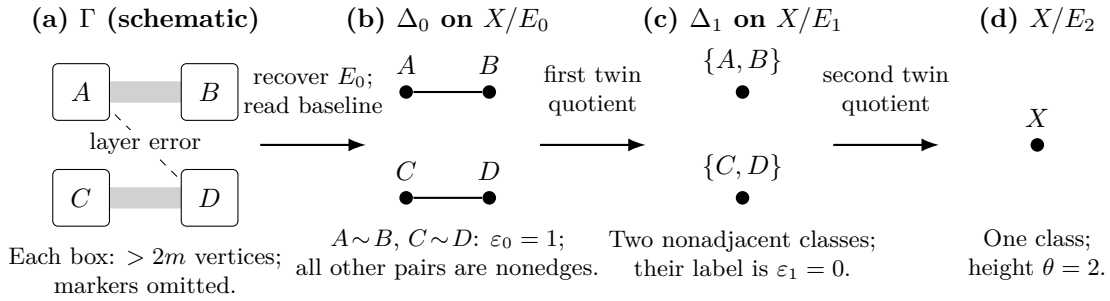
\begin{figure}[!htbp]
\centering
\begin{tikzpicture}[x=1cm,y=1cm,>=Latex,every node/.style={font=\small},
  cls/.style={draw,rounded corners=2pt,fill=white,minimum width=.75cm,minimum height=.75cm},
  dot/.style={circle,fill=black,inner sep=1.8pt}]
  \path[use as bounding box] (-.60,-1.35) rectangle (14.50,2.60);
  \node[font=\small\bfseries] at (1.20,2.35) {(a) $\Gamma$ (schematic)};
  \draw[line width=8pt,black!18] (.50,1.40)--(1.90,1.40);
  \draw[line width=8pt,black!18] (.50,0)--(1.90,0);
  \draw[dashed] (.50,1.40)--(1.90,0);
  \node[fill=white,inner sep=1pt,font=\footnotesize] at (1.20,.72) {layer error};
  \foreach \name/\xx/\yy in {A/.35/1.40,B/2.05/1.40,C/.35/0,D/2.05/0}{
    \node[cls] (a\name) at (\xx,\yy) {$\name$};
  }
  \node[align=center,font=\footnotesize,anchor=north] at (1.20,-.52)
    {Each box: $>2m$ vertices;\\markers omitted.};
  \draw[->,thick] (2.72,.70)--(4.12,.70);
  \node[align=center,font=\footnotesize] at (3.42,1.42) {recover $E_0$;\\read baseline};
  \node[font=\small\bfseries] at (5.20,2.35) {(b) $\Delta_0$ on $X/E_0$};
  \foreach \name/\xx/\yy in {A/4.65/1.40,B/5.75/1.40,C/4.65/0,D/5.75/0}{
    \node[dot,label=above:$\name$] (b\name) at (\xx,\yy) {};
  }
  \draw[thick] (bA)--(bB); \draw[thick] (bC)--(bD);
  \node[align=center,font=\footnotesize] at (5.20,-.75)
    {$A\!\sim\!B$, $C\!\sim\!D$: $\eps_0=1$;\\all other pairs are nonedges.};
  \draw[->,thick] (6.42,.70)--(7.82,.70);
  \node[align=center,font=\footnotesize] at (7.12,1.42) {first twin\\quotient};
  \node[font=\small\bfseries] at (9.10,2.35) {(c) $\Delta_1$ on $X/E_1$};
  \node[dot,label={above:{$\{A,B\}$}}] at (9.10,1.40) {};
  \node[dot,label={above:{$\{C,D\}$}}] at (9.10,0) {};
  \node[align=center,font=\footnotesize] at (9.10,-.75)
    {Two nonadjacent classes;\\their label is $\eps_1=0$.};
  \draw[->,thick] (10.30,.70)--(11.70,.70);
  \node[align=center,font=\footnotesize] at (11,1.42) {second twin\\quotient};
  \node[font=\small\bfseries] at (13.0,2.35) {(d) $X/E_2$};
  \node[dot,label=above:$X$] at (13.0,.70) {};
  \node[align=center,font=\footnotesize] at (13.0,-.75) {One class;\\height $\theta=2$.};
\end{tikzpicture}
\caption{Notation is from Section~\ref{sec:recovery}; the figure is meant to be read with Section~\ref{sec:proofidea} and reread after Theorem~\ref{thm:recovery}. Recovery in Theorem~\ref{thm:recovery}, shown at height two with $\eps_0=1$ and $\eps_1=0$. Panel (a) abbreviates large initial classes; the gray bands denote nearly complete joins, and the dashed segment symbolizes a layer error. Robust-marker components first recover $E_0$; the at-most-$m$ neighbor/nonneighbor test then gives the \emph{unperturbed} baseline in (b). Twin classes in that baseline, not twins of the original perturbed graph, give (c) and (d). Here twins are understood as in Theorem~\ref{thm:recovery}: the mutual adjacency of the pair is ignored. The finite diagram illustrates the successor operation only; the proof treats limit stages separately.}
\label{fig:recovery}
\end{figure}

The recovered core cosets carry copies of the same GRR, so every automorphism $F$ satisfies $F(hx)=hF(x)$ for $h\in H$. If $F$ fixes $1$, the recovered layers and baseline bits also identify each packet $W_\alpha$ as the neighbors, or the nonneighbors, of $1$ in that layer. The packets are chosen so that this exact $H$-equivariance and preservation of $W_\alpha$ force $F(a_\alpha)=a_\alpha$ (Lemmas~\ref{lem:packet} and~\ref{lem:booleanpacket}). Applying this conclusion to the rooted automorphism $F_x(y)=F(yx)F(x)^{-1}$ gives $F(a_\alpha x)=a_\alpha F(x)$ at every base point $x$. These identities and $H$-equivariance propagate along words, forcing $F$ to be the identity. Hence the full automorphism group consists of right translations.

To construct the countable core, Theorem~\ref{thm:countable} assigns distinct finite common-neighbor counts to the primary inverse pairs of Lemma~\ref{lem:countable}. Every remaining nonidentity element is a square of a primary element and receives infinitely many representations; its inverse-pair relation is recovered by two primary steps (Lemma~\ref{lem:recognition}). Finite quadratic avoidance (Lemma~\ref{lem:BP}) permits the protected seed insertions of Lemma~\ref{lem:seedinsert} while preserving the frozen counts. Each insertion is completed by covering the newly selected elements with protected robust markers (Lemma~\ref{lem:markerprescription}) before the next primary count is frozen, and later insertions preserve the frozen counts and the earlier markers. The resulting graph has automorphism group $\Xi(H)$, which consists of right translations by Theorem~\ref{thm:inversepairindex}, and has the required connected robust markers. If $H=G$, no extension is needed.

Theorem~\ref{thm:countable} also realizes $\Xi(G)$ for countable exceptional groups. Uncountable exceptional groups use the finite inverse-pair coloring proved in Appendix~\ref{app:finitecolor}, encoded by codegree bands or by a finite Boolean factor already present in the group, to recover exactly $\Xi(G)$ (Section~\ref{sec:exceptions}).

For abundance, we vary intrinsically recoverable data. In countable groups, each bit decides whether one further marked insertion raises a primary count before it is frozen, and the strictly increasing sequence of frozen counts recovers the binary sequence (Theorem~\ref{thm:countable}). In uncountable regular groups, the free bits at zero and nonzero limit indices give $2^\kappa$ choices recovered with the hierarchy (Theorem~\ref{thm:cube}). In uncountable exceptional groups, the decoder recovers a binary refinement of the finite coloring (Proposition~\ref{prop:exceptmany}). For all countable groups and for uncountable exceptional groups, complementation preserves the coded graphs' automorphism groups and nonisomorphism, and the construction supplies $\kappa$ common neighbors at each distinct pair in the complements. For uncountable regular groups, the dense layers supply these witnesses directly. The final graphs therefore have diameter two, as required in Theorem~\ref{thm:all}.

\subsection{Conventions, notation, and dependencies}\label{sec:notation}
We work in ZFC; all graphs are simple and undirected. A permutation or graph
automorphism is \emph{rooted} if it fixes the identity vertex (the vertex $0$ in
additive notation). For a connection set
$S\subseteq G\setminus\{1\}$ our right Cayley convention is
\[
x\sim y\quad\Longleftrightarrow\quad xy^{-1}\in S.
\]
Writing $R(G)$ for the prescribed right translations gives
\begin{equation}\label{eq:sandwich}
R(G)\le\Xi(G)\le\Aut(\Cay(G,S)).
\end{equation}
For $g\ne1$, define
\[
\nu_S(g)=|\{(u,v)\in S^2:u^{-1}v=g\}|=|S\cap Sg|.
\]
We call $\nu_S(g)$ the \emph{codegree of the difference $g$} with respect to $S$.
The common-neighbor formula in this convention is
\[
\CN(x,y)=|Sx\cap Sy|=|S\cap Syx^{-1}|
=\nu_S(yx^{-1})=\nu_S(xy^{-1});
\]
the last equality follows by interchanging the endpoints of a representation.
Several references use edges $g^{-1}h\in S$ and the left regular action. Inversion
$g\mapsto g^{-1}$ transports that convention to ours and conjugates the left
regular group to $R(G)$. The order-four convention for generalized dicyclic
groups agrees with \cite[Introduction]{LdSJAC}; it does not include split
generalized dihedral groups with an involutory inverter.

\paragraph{Notation.}
For a set $X$, let $[X]^2$ be its set of two-element subsets. If
$\mathcal F\subseteq[X]^2$ is an edge set, write
\[
\maxdeg(X,\mathcal F)=\sup_{x\in X}|\{y\in X:\{x,y\}\in \mathcal F\}|
\]
for the supremum of its vertex degrees, with value $0$ on the empty graph.
For a subset $P$ of a group, write $\Diff(P)=P^{-1}P\setminus\{1\}$ instead.
The symbols $\Delta_\alpha$ are reserved for the baseline quotient graphs
in Section~\ref{sec:recovery}; its layer-error relations are denoted by $\mathcal L_\alpha$. For a group $G$, put
\[
\mathcal I(G)=\{\invp{g}:g\in G\setminus\{1\}\},\qquad
\mathcal X_G=\{0,1\}^{\mathcal I(G)}.
\]
We identify a connection set with its indicator on inverse pairs and give $\mathcal X_G$
the product topology of the discrete two-point factors. A basic cylinder
prescribes membership of finitely many inverse pairs.
For an ordinal $\theta$, $\F_2^{(\theta)}$ is the additive group of finitely
supported functions $\theta\to\F_2$, and
$\operatorname{supp}(v)=\{\alpha<\theta:v(\alpha)\ne0\}$.
For an abelian group written additively,
$A[2]=\{a\in A:2a=0\}$ and $2A=\{2a:a\in A\}$.
For an infinite cardinal $\kappa$, $\cf(\kappa)$ is the least cardinality of
an unbounded subset of its initial ordinal.
The complement of a graph $\Gamma$ is denoted by $\overline\Gamma$ (or
$\Gamma^c$); the complementary connection set is
$\overline S=(G\setminus\{1\})\setminus S$.

\paragraph{Dependencies and scope.}
The regular cases use finite quadratic avoidance \cite[Theorem~1]{BP} and the inverse-pair
classification in Theorem~\ref{thm:inversepairindex}. All countable-core,
packet, and hierarchy arguments are included. Theorem~\ref{thm:countable} supplies countably infinite robust cores by a protected-marker schedule, so the regular cases require no asymptotic finite-GRR input. For the exceptional families, Appendix~\ref{app:finitecolor}
proves Theorem~\ref{thm:finitecolor}, which needs at most $17$ colors, by adapting the construction of \cite{FC} and using the published
short-word theorem \cite[Theorem~8]{LdS} and finite-color compactness
\cite{dBE}. Thus no result from that preprint is left as a separate unproved
premise. The regular-case construction does not depend on this appendix.

\subsection{Inverse-pair classification and the Cayley-index deduction}\label{sec:indexinput}
\begin{theorem}[Leemann--de la Salle, inverse-pair index]
\label{thm:inversepairindex}
For every infinite group $G$,
\[
[\Xi(G):R(G)]=
\begin{cases}
1, & G\text{ is Boolean, or is neither abelian nor generalized dicyclic};\\[2pt]
2, & \begin{array}[t]{@{}l@{}}
G\text{ is non-Boolean abelian, or is generalized dicyclic}\\
\text{and not isomorphic to }Q_8\times B\text{ for any Boolean group }B;
\end{array}\\[2pt]
8, & G\cong Q_8\times B\text{ for an infinite Boolean group }B.
\end{cases}
\]
\end{theorem}
This is the numerical form of \cite[Theorem~2]{LdS}. Its nonabelian input is
\cite[Theorem~7]{LdSJAC}, which does not assume finite generation. The index
is the order of the identity stabilizer in $\Xi(G)$, since $R(G)$ acts
regularly.

\begin{proof}[Deduction of Theorem~\ref{thm:cayleyindex} from Theorem~\ref{thm:all}]
By \eqref{eq:sandwich} and Theorem~\ref{thm:inversepairindex}, every Cayley graph
has index at least the displayed value of $[\Xi(G):R(G)]$. Theorem~\ref{thm:all}
supplies a connected realization of $\Xi(G)$, attaining that value.
\end{proof}

\subsection{The historical formulation and the scope of the result}
We use the formulation in Leemann--de la Salle
\cite[Conjecture~1 in the introduction]{LdSJAC}, who attribute it to Watkins's
1976 free-products paper \cite{Watkins76}. Their statement begins
``There is an integer $n$ such that every group of cardinality at least $n$''
and requires a GRR after excluding the two algebraic families in
Theorem~\ref{thm:watkins}. In full mathematical form, their conjecture is
\[
\exists n\in\N\quad\forall G\text{ in a regular case}:\qquad
 |G|\ge n\ \Longrightarrow\ G\text{ admits a GRR}.
\]
The thirteen
additional finite exceptions have order at most $32$
\cite[Introduction]{LdS}; therefore the finite classification together with
Theorem~\ref{thm:watkins} gives $n=33$ in this formulation.

The finite cases were settled by Godsil, Hetzel, and their predecessors
\cite{Imrich69,Watkins71,NW1,NW2,Imrich75,IW,Hetzel,Godsil}; Leemann and
de la Salle settled the infinite finitely generated cases with locally finite
Cayley graphs \cite{LdSAHL}; Leemann has announced a program concerning the
countable case and small generating sets \cite{LeemannTalk}; the scope of the
theorem proved here is arbitrary infinite cardinality. That program has not, at
the time of writing, been published as a complete proof. The announcement is
acknowledged as relevant overlapping work, not used as a proof input.

\paragraph{Valence and generating-set refinements.}
The graphs constructed here have valence $|G|$, which is unavoidable for a
connected Cayley graph of a non-finitely-generated group. Locally finite GRRs
for finitely generated groups are the subject of \cite{LdSAHL}; refinements
controlling the generating set are not treated here. The announced ``small
generating set'' program is described in \cite[seminar announcement]{LeemannTalk}.

Theorem~\ref{thm:cayleyindex} also establishes the infinite-group part of
Leemann--de la Salle's later Cayley-index conjecture
\cite[Conjecture~3]{LdS}. 
\subsection{Related representation and reconstruction results}
Frucht's realization of finite groups as graph automorphism groups \cite{Frucht}
and Sabidussi's characterization through regular graph actions \cite{Sabidussi}
provide the background. Here the prescribed group is itself the vertex set,
and the prescribed translations are retained. Morris and Tymburski completed
the finite minimum-index problem \cite{MT}. Babai's directed representations
\cite{Babai78} address a different problem, since forgetting directions can add
automorphisms; Watkins's free-product construction \cite{Watkins76} is an earlier
infinite undirected result. Leemann and de la Salle developed orientation- and
color-rigidity methods \cite{LdSJAC}, and determined the remaining finitely
generated Cayley indices \cite[Theorem~1]{LdS}.

Byrne, Donner, and Sibley computed the complete inverse-pair symmetry group for
finite groups \cite{BDS}; the arbitrary-group classification used here is
Theorem~\ref{thm:inversepairindex}. In the quaternion--Boolean case the rooted
group has eight elements, corresponding to independent signs on the three
quaternion directions. The Boolean-factor identity in
Section~\ref{sec:exceptions} is also \cite[Lemma~4]{LdS}; see also the
color-preserving Cayley-automorphism literature \cite{CCA}.

Babai and Godsil initiated the asymptotic GRR program \cite{BG82}; its finite
abelian and generalized-dicyclic counterparts were studied in \cite{DSV,MSV}.
Xia and Zheng's uniform theorem \cite[Theorem~1.3]{XZ} completes the finite asymptotic program. It is cited here for context, not used in the countable-core construction.

The author's invariant-metric constructions \cite{Reflection,FC} retain explicit
distance classes. The present argument reconstructs distinctions from adjacency
alone. Without its exception packets, the construction is a substitution over an
alternating join--union hierarchy; at finite depth its baseline is described by a
cotree \cite{Cographs}. The packets can destroy the module property of the core
copies, since an exterior vertex may distinguish two vertices in one class, but
Theorem~\ref{thm:recovery} still recovers the hierarchy and the packets then
synchronize the local actions. At limits the proof uses continuous equivalence
relations.

\paragraph{Organization.}
Section~\ref{sec:cores} combines finite avoidance with protected markers and proves the countable theorem, Theorem~\ref{thm:countable}.
Section~\ref{sec:recovery} proves the group-free reconstruction theorem and treats finite packets.
Section~\ref{sec:main} proves the prescribed uncountable cube, Theorem~\ref{thm:cube}, Theorem~\ref{thm:watkins} and the uncountable
regular cases of Theorem~\ref{thm:all}. Section~\ref{sec:exceptions} completes
Theorem~\ref{thm:all} using the finite-color result proved in
Appendix~\ref{app:finitecolor}. Section~\ref{sec:consequences} gives the
finite-prescription and algebraic consequences, and the section after it records computational checks and further questions.

\section{Avoidance, protected markers and the countable theorem}\label{sec:cores}
This section works in an infinite group $H$ and ends with the countable theorem. Every nonidentity difference is primary or secondary (Lemma~\ref{lem:countable}); a seed block for a target difference avoids every frozen primary difference, so it can be inserted without changing their counts (Lemma~\ref{lem:seedinsert}), and protected markers then cover the newly selected differences by finite induced $m$-robust subgraphs without disturbing those counts (Lemmas~\ref{lem:star} and~\ref{lem:markerprescription}). Theorem~\ref{thm:countable} interleaves these finite operations to code a binary sequence in the frozen primary counts of a countable group, giving a closed Cantor family of connection sets that realize $\Xi(H)$ under any finite prescription.

\subsection{Primary and secondary differences and seed insertion}\label{sec:countable}
We first record the exact avoidance input. A group monomial of degree at most two contains at most two occurrences of $X$ or $X^{-1}$; coefficients do not count toward its degree.
The \emph{second Zariski topology} on $H$ has as an open subbasis the cozero
sets $\{h\in H:w(h)\ne1\}$ of all such monomials, including constants and
degree-one words. Equivalently, the solution sets of
$a_0X^{\eta_1}a_1X^{\eta_2}a_2=1$ and their lower-degree analogues form a closed
subbasis, where the coefficients lie in $H$ and $\eta_1,\eta_2\in\{1,-1\}$
are chosen independently; mixed exponents are included \cite{BP}.
\begin{lemma}[Finite quadratic avoidance]\label{lem:BP}
In any infinite group, finitely many degree-at-most-two inequalities $w_i(X)\ne1$, all true at $X=1$, have infinitely many simultaneous solutions.
\end{lemma}
\begin{proof}
Banakh--Protasov \cite[Theorem~1]{BP} prove that the second Zariski topology is nondiscrete. The simultaneous cozero set is an open neighborhood of 1, so it remains only to see that such a neighborhood cannot be finite. If a nonempty open set were finite, deleting its other points would isolate one point, and translation would make the topology discrete. The simultaneous solution set is therefore infinite.
\end{proof}

For a group $H$ put $\operatorname{Sq}(H)=\{x^2:x\in H\}$ and
\[
B=\{x^2:x\notin\operatorname{Sq}(H)\}\setminus\{1\},\qquad A_*=(H\setminus\{1\})\setminus B;
\]
the elements of $B$ and their inverse pairs are \emph{secondary}, those of $A_*$ are \emph{primary}. A \emph{representation} of $h$ in a set $A$ is a pair $(u,v)\in A\times A$ with $u^{-1}v=h$, and the \emph{count} $\nu_A(h)$ is their number; for symmetric $A$ it is the same for $h$ and $h^{-1}$. For finite $P\subseteq H$ put $\Diff(P)=\{p^{-1}q:p,q\in P\}\setminus\{1\}$. The \emph{seed} of an inverse pair $\invp{h}$, its \emph{target}, is $P=\{r,r^{-1}\}$ for an $r$ fixed with $r^2=h$ if $h\in\operatorname{Sq}(H)$, and $P=\{1,h,h^{-1}\}$ otherwise. For symmetric $P$ and $t\in H$ the \emph{block} is $Y=tP\cup Pt^{-1}$, and \emph{inserting} it into a selected set $A$ gives $A\cup Y$.

The partition is designed so that each seed creates only its own primary inverse pair: any further difference it creates falls in the secondary class, where infinite codegree is allowed. In additive $\mathbb Z$ the secondary elements are those $\equiv2\pmod4$, and the seed $\{0,1,-1\}$ of the primary difference $1$ gives the differences $\pm1$ and $\pm2$; the extra pair $\pm2$ is secondary.

\begin{lemma}[Primary and secondary pairs and their seeds]\label{lem:countable}
Let $H$ be an infinite group. Every nonidentity inverse pair is primary or secondary, and not both; every secondary $b$ is $a^2$ for a primary nonsquare $a$; there are infinitely many primary inverse pairs. The seed $P$ of a nonidentity (primary or secondary) pair $\invp{h}$ has $P=P^{-1}$, $|P|\le3$ and $h\in\Diff(P)$, with $\Diff(P)=\invp{h}$ if $h$ is a square and $\Diff(P)\subseteq\invp{h}\cup B$ otherwise.
\end{lemma}
\begin{proof}
The sets $B$ and $A_*$ are inverse-closed and partition $H\setminus\{1\}$, so every nonidentity inverse pair is primary or secondary, and not both. Every secondary $b$ is $a^2\ne1$ for a primary nonsquare $a$. There are infinitely many primary inverse pairs: otherwise their elements and the squares of their nonsquare elements would make $H$ finite.

The seed is chosen so that the target difference appears while already protected primary differences do not. For a square $h=r^2$, primary or secondary, the seed $P=\{r,r^{-1}\}$ has $\Diff(P)=\invp{h}$, and $r\ne r^{-1}$ as $h\ne1$. For a primary nonsquare $a$, the seed $P=\{1,a,a^{-1}\}$ has $\Diff(P)=\{a^{\pm1},a^{\pm2}\}\setminus\{1\}\subseteq\invp{a}\cup B$, since $a^2\in B\cup\{1\}$. Every secondary $b$ is a square, so its seed is of the first kind. Every seed is symmetric, has at most three elements and contains its target in $\Diff(P)$.

\end{proof}

\begin{lemma}[Protected seed insertion]\label{lem:seedinsert}
Let $H$ be an infinite group, let $A,R\subseteq H\setminus\{1\}$ be finite, symmetric and disjoint (selected and forbidden), let $D=D^{-1}\subseteq H\setminus\{1\}$ be finite (protected), and let $P=P^{-1}\subseteq H$ be finite and nonempty. For $t\in H$ let $Y=tP\cup Pt^{-1}$ (the \emph{block}). The \emph{seed inequalities} are
\[
 (tp)^{-\sigma}(tq)^{\sigma'}\ne d\qquad(p,q\in P,\ \sigma,\sigma'\in\{1,-1\},\ d\in D\setminus\Diff(P))
\]
and the \emph{freshness conditions} are $Y\cap(A\cup R\cup\{1\})=\varnothing$ and $u^{-1}v\notin D$ whenever one of $u,v$ lies in $A$ and the other in $Y$ (no old--new representation of $D$). For every finite $E\subseteq H$ and every finite set of further inequalities of degree at most two in $t$ that hold at $t=1$, some $t\notin E$ satisfies these, the seed inequalities and the freshness conditions. For such $t$ and $A'=A\cup Y$: $A'\cap R=\varnothing$; $\nu_{A'}(d)=\nu_A(d)$ for $d\in D\setminus\Diff(P)$; $\nu_{A'}(d)\le\nu_A(d)+|Y|$ for $d\in D$; $\nu_{A'}(h)\ge\nu_A(h)+1$ for $h\in\Diff(P)$; and $|Y|\le6$ if $|P|\le3$. 
\end{lemma}

\begin{proof}
Choose
\[
Y=tP\cup(tP)^{-1}=tP\cup Pt^{-1}
\]
disjoint from $A\cup R\cup\{1\}$, with no old--new protected differences and no new--new differences in $D\setminus\Diff(P)$. For the latter conditions impose the seed inequalities. These inequalities have degree at most two; for $(\sigma,\sigma')=(1,1)$
the left-hand side is the constant $p^{-1}q$. To invoke finite quadratic avoidance, evaluate the remaining conditions at $t=1$; the four sign choices give
\[
 \left.(tp)^{-\sigma}(tq)^{\sigma'}\right|_{t=1}
 =\begin{cases}
 p^{-1}q,&(\sigma,\sigma')=(1,1),\\
 pq,&(\sigma,\sigma')=(-1,1),\\
 p^{-1}q^{-1},&(\sigma,\sigma')=(1,-1),\\
 pq^{-1},&(\sigma,\sigma')=(-1,-1).
 \end{cases}
\]
Since $P=P^{-1}$, all four values belong to $\Diff(P)\cup\{1\}$ and therefore
differ from $d\in D\setminus\Diff(P)$. Lemma~\ref{lem:BP} therefore gives infinitely
many simultaneous solutions, after which freshness, avoidance of $R$, and the old--new conditions exclude only finitely many additional $t$: each failure says that $tp$ or $pt^{-1}$ lies in a finite set, or that $u^{-1}v=d\in D$ with one of $u,v$ old and the other in $Y$, an equation of degree one in $t$ with one solution for each choice of $p$, the old element and $d$. The same holds after adding finitely many further inequalities of degree at most two that hold at $t=1$, and $t\notin E$ for a finite $E$ excludes finitely many more $t$.
What matters for the target difference is that, if $p\ne q$, then $tp\ne tq$, and each intended difference $p^{-1}q$ gains the new representation $(tp)^{-1}(tq)$. For such $t$ a new representation of $d\in D$ is not old--new, so it is new--new; the seed inequalities exclude these for $d\notin\Diff(P)$, and the first entry of each determines the second. Hence $\nu_{A\cup Y}(d)=\nu_A(d)$ for $d\in D\setminus\Diff(P)$ and $\nu_{A\cup Y}(d)\le\nu_A(d)+|Y|$ for $d\in D$. Also $A\cup Y$ is disjoint from $R$, and $|Y|\le2|P|$. So a seed can be inserted again into the enlarged data, each insertion raising the count of its target by at least one.
\end{proof}

\begin{lemma}[Codegree recognition]\label{lem:recognition}
Let $H$ be a countably infinite group. If $U\subseteq H\setminus\{1\}$ is symmetric, $\nu_U(a)$ is finite and positive for every primary $a$, distinct primary pairs have distinct counts, and $\nu_U(b)=\infty$ for every secondary $b$, then $\Aut(\Cay(H,U))=\Xi(H)$ and $\nu_U(h)>0$ for all $h\ne1$.
\end{lemma}

\begin{proof}
The finite, distinct primary counts of $U$ identify each primary inverse pair, so every graph automorphism satisfies $F(ax)\in\{a,a^{-1}\}F(x)$ for primary $a$. It remains to recover a secondary pair from a primary square root. For $b=a^2\in B$, apply this at $x$ and $ax$ to obtain
\[
F(bx)F(x)^{-1}\in\{a^2,1,a^{-2}\}.
\]
Injectivity excludes $1$, leaving exactly the required inverse pair. All inverse-pair differences are therefore preserved, so $\Aut(\Cay(H,U))\subseteq\Xi(H)$. Conversely, if a permutation $F$ of $H$ satisfies $F(hx)\in\{h,h^{-1}\}F(x)$ for all $h,x\in H$, then $F(y)F(x)^{-1}\in\{yx^{-1},xy^{-1}\}$ for all $x,y$, and since $U=U^{-1}$ this lies in $U$ exactly when $yx^{-1}$ does; so $F$ keeps edges and nonedges, and $\Aut(\Cay(H,U))=\Xi(H)$. Every nonidentity difference has positive codegree, which supplies generation as well as the stated positivity.
\end{proof}

\subsection{Protected markers}\label{sec:robust}

\begin{definition}[Robust graphs]\label{def:robustgraph}
Throughout, $m\ge0$ denotes an integer. A \emph{cut} of a graph on $X$ is a partition
$X=Y\sqcup Z$ into two nonempty parts. A \emph{crossing neighbor} of a vertex
lies in the opposite part and is adjacent to it; a \emph{crossing nonneighbor}
lies in the opposite part and is not adjacent to it. The \emph{crossing degree} and
\emph{crossing nondegree} of a vertex are respectively the numbers of its crossing
neighbors and crossing nonneighbors. A cut in which every vertex has at most $m$
crossing neighbors is an \emph{$m$-cut}. A finite simple graph
with at least two vertices is $m$-robust if every cut has a vertex with more than $m$ crossing neighbors and a vertex with more than $m$ crossing nonneighbors. The vertices can differ.
\end{definition}

\begin{lemma}[Robust unions]\label{lem:robustunion}
Let $\Gamma$ be a graph on $X$. Call an edge of $\Gamma$ a \emph{robust-marker edge} if it lies in a finite induced $m$-robust subgraph, and let $M_m$ be the set of these edges.
\begin{enumerate}[label=(\roman*),nosep]
\item If $F_1,\ldots,F_r$ induce finite $m$-robust subgraphs and their intersection graph is connected, then the graph induced on $F_1\cup\cdots\cup F_r$ is $m$-robust. Here the intersection graph has vertices $1,\ldots,r$, with $j$ and $k$ adjacent when $F_j\cap F_k\ne\emptyset$.
\item The relation
\[
 x\sim_m y
 \quad\Longleftrightarrow\quad
 x=y\ \text{or some finite induced $m$-robust subgraph contains $x,y$}
\]
is an equivalence relation, and its classes are exactly the components of $(X,M_m)$.
\item If $X$ has at least two vertices, then $(X,M_m)$ is connected if and only if every finite vertex set is contained in one finite induced $m$-robust subgraph.
\item A finite graph is $m$-robust if and only if its complement is. Hence $\Gamma$ and its complement have the same finite induced $m$-robust subgraphs, as vertex sets, and the same relation $\sim_m$.
\end{enumerate}
\end{lemma}

\begin{proof}
(i) For a cut of the union, either some $F_j$ is split, or no $F_j$ is split. In the first case, robustness of that $F_j$ supplies more than $m$ crossing neighbors and more than $m$ crossing nonneighbors; those witnesses remain valid in the larger union, since the subgraphs are induced. In the second case, intersecting members must lie on the same side. Connectedness of the intersection graph then forces the whole union onto one side, a contradiction.

(ii) Reflexivity and symmetry are clear. Transitivity follows by taking the union of two witnesses through the common vertex, which is $m$-robust by (i) with $r=2$. Every robust graph is connected: a component of a disconnected one and its complement would form a cut without crossing neighbors. Hence two vertices of one witness are joined by a path of robust-marker edges inside it; conversely, the endpoints of a robust-marker edge are $\sim_m$-related, so every component of $(X,M_m)$ lies in one class. Thus the classes are exactly the components.

(iii) If every finite vertex set lies in one finite induced $m$-robust subgraph, then so does every pair of vertices, and (ii) puts them in one component. Conversely, let $(X,M_m)$ be connected and let $A\subseteq X$ be finite; enlarging $A$, we may assume it has at least two points. Fix $a\in A$, connect $a$ to every other point of $A$ by a path in $(X,M_m)$, and choose for each edge of these paths a finite induced $m$-robust subgraph containing it. Markers of consecutive edges of a path meet, and the markers of the first edges all contain $a$, so the intersection graph of these finitely many markers is connected. By (i), the union of the markers along those paths is a finite induced $m$-robust subgraph containing $A$.

(iv) Robustness is invariant under complementation because crossing neighbors and crossing nonneighbors exchange roles. An induced subgraph of the complement is the complement of the induced subgraph, so the same finite witnesses and the same relation $\sim_m$ work in a graph and its complement, with the same $m$.
\end{proof}

\begin{lemma}[A robust four-block pattern]\label{lem:p4}
Partition a finite vertex set into four blocks $B_1,B_2,B_3,B_4$, each of size greater than $2m$. Put all edges between consecutive blocks and no edges between nonconsecutive blocks. Edges within blocks are arbitrary. The resulting graph is $m$-robust.
\end{lemma}
\begin{proof}
Suppose a cut has crossing degree at most $m$ everywhere. To obtain the required crossing witness, first observe that the cut cannot split a block: if it did, one vertex on each side would force each side of a neighboring block to have at most $m$ points, contradicting that block's size. No block is therefore split, and connectivity of the path on the four blocks forces some consecutive blocks onto opposite sides, giving crossing degree greater than $m$. For the crossing-nonneighbor witness, apply the same argument to the complement: the nonconsecutive-block graph is again a connected four-vertex path and has no isolated block. \end{proof}

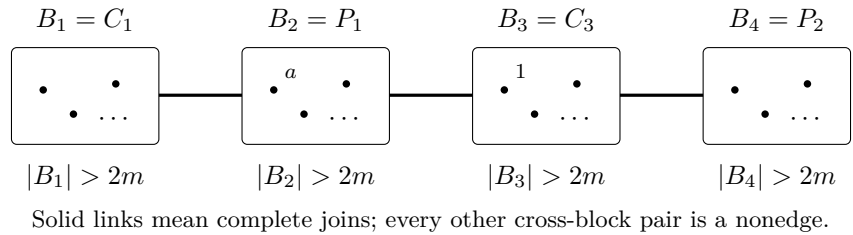
\begin{figure}[!htbp]
\centering
\begin{tikzpicture}[x=1cm,y=1cm, every node/.style={font=\small}]
  \foreach \x/\name/\real in {0/B_1/C_1,3.05/B_2/P_1,6.1/B_3/C_3,9.15/B_4/P_2}{
    \draw[rounded corners=2pt] (\x,0) rectangle +(1.95,1.28);
    \node at (\x+.975,1.65) {$\name=\real$};
    \foreach \a/\b in {.42/.72,.82/.4,1.38/.8}{\fill (\x+\a,\b) circle (1.4pt);}
    \node at (\x+1.38,.34) {$\cdots$};
    \node at (\x+.975,-.43) {$|\name|>2m$};
  }
  \foreach \x in {1.95,5,8.05}{\draw[line width=1.2pt] (\x,.65)--+(1.1,0);}
  \node[font=\scriptsize,anchor=south west] at (3.48,.74) {$a$};
  \node[font=\scriptsize,anchor=south west] at (6.53,.74) {$1$};
  \node[font=\footnotesize] at (5.55,-1.02) {Solid links mean complete joins; every other cross-block pair is a nonedge.};
\end{tikzpicture}
\caption{The four-block robust marker. Edges within each block are arbitrary. In the protected-marker construction, $1\in C_3$ and the marked generator lies in $P_1$.}
\label{fig:robust}
\end{figure}

\begin{lemma}[Protected star insertion]\label{lem:star}
Let $A,R\subseteq H\setminus\{1\}$ be finite, symmetric, and disjoint in an infinite group $H$, and let $D=D^{-1}\subseteq H\setminus\{1\}$ be finite. Suppose $P=P^{-1}\subseteq A$ is nonempty,
\[
 \Diff(P)\cap D=\varnothing,
\]
and $Z$ is a finite set disjoint from $P$, possibly containing $1$. One can choose a new vertex $t$, outside any prescribed finite vertex set, and set
\[
 Y=tP\cup(tP)^{-1},\qquad
 A'=A\cup Y,\qquad
 R'=R\cup\bigcup_{z\in Z}\invp{tz^{-1}},
\]
so that $Y\cap(A\cup R\cup\{1\})=\varnothing$, $A'\cap R'=\varnothing$, and
\[
 \nu_{A'}(d)=\nu_A(d)\qquad\text{for all }d\in D.
\]
In this partial Cayley specification, $t$ is adjacent to $P$ and nonadjacent to $Z$.
\end{lemma}
\begin{proof}
Use Lemma~\ref{lem:seedinsert} with the given $A$, $R$, $D$ and $P$. The aim is to add the prescribed star while leaving every protected codegree unchanged. Since $\Diff(P)\cap D=\varnothing$, all new--new representations of differences in $D$ can be prohibited, and all old--new representations are excluded by finitely many additional values of $t$. The seed inequalities of that lemma hold at $t=1$.
For each $z\in Z$, putting $tz^{-1}$ outside old $A$ excludes finitely many $t$. We must also keep each reserved nonedge away from the new selected block. It cannot equal $tp$ with $p\in P$, since $P=P^{-1}$ and $Z\cap P=\varnothing$. Equality to $(tp)^{-1}$ is equivalent to
\[
 tz^{-1}tp=1.
\]
Its negation is also true at $t=1$, since $z\ne p$. Lemma~\ref{lem:seedinsert}, applied with these further inequalities and with $E$ the prescribed finite vertex set together with the finitely many $t$ excluded above, gives $t$ satisfying them, the seed inequalities and the freshness conditions. The resulting assignments are inverse-compatible, so the displayed extensions have the stated properties and the protected counts are preserved, by that lemma with $D\setminus\Diff(P)=D$.
\end{proof}

\begin{lemma}[Protected finite marker]\label{lem:markerprescription}
Let $A,R\subseteq H\setminus\{1\}$ be finite, symmetric, and disjoint in an infinite group $H$. Let $D=D^{-1}\subseteq H\setminus\{1\}$ be finite. If $a\in A$ and $a^2\notin D$, there are finite symmetric extensions $A\subseteq A'$, $R\subseteq R'$ retaining disjointness, preserving every $\nu_A(d)$ for $d\in D$, and prescribing a finite induced $m$-robust subgraph containing the edge $\{1,a\}$. Every member of $(A'\setminus A)\cup\invp{a}$ is an oriented edge difference in that marker. Every later completion containing $A'$ and avoiding $R'$ retains its robustness.
\end{lemma}

Plan: grow two symmetric pools $P_1\supseteq\invp{a}$ and $P_2$ inside $A''$, with every cross-pair reserved as a nonedge, then attach the star blocks $C_1$ and $C_3$ by Lemma~\ref{lem:star}, so that the four blocks have the cross-pattern of Lemma~\ref{lem:p4}.
\begin{proof}
Put $L=2m+1$. Start a symmetric pool $P_1$ with $\invp{a}$; its nonidentity differences lie in $\invp{a^2}$ and therefore avoid $D$. Grow it to size $L$ or $L+1$, and build a disjoint symmetric pool $P_2$ of size $L$ or $L+1$, selecting each new inverse pair $\invp{x}$ into the current $A$ and outside the current $R$ (the given $R$ and the pairs reserved at earlier insertions into $P_2$); this excludes only finitely many more values of $x$.
At each pool insertion prohibit new representations of $D$. Differences between a new point $x^{\pm1}$ and an old selected point exclude finitely many values of $x$. Differences within $\invp{x}$ are controlled by $x^2\notin D$; these are quadratic inequalities true at $x=1$. For $P_2$, also prescribe all cross-pairs with $P_1$ as nonedges. The reserved inverse pairs $\invp{xp^{-1}}$ and $\invp{x^{-1}p^{-1}}$ must avoid
the newly selected pair $\invp{x}$. Their coincidences with $\invp{x}$ reduce to $p=1$
or to
\[
 xp^{-1}x=1\quad\hbox{or}\quad x^{-1}p^{-1}x^{-1}=1,
 \qquad p\in P_1.
\]
Both equations are false at $x=1$. Lemma~\ref{lem:BP} therefore supplies each pool choice together with the square inequalities. The reserved pairs must also avoid the current selected set $A_{\mathrm{cur}}$ (the given $A$, $P_1$ and the earlier points of $P_2$); as $A_{\mathrm{cur}}$ is symmetric, this is the finite exclusion $x\notin A_{\mathrm{cur}}p\cup p^{-1}A_{\mathrm{cur}}$ for $p\in P_1$, which the infinitely many solutions of Lemma~\ref{lem:BP} absorb. Reserve all cross-differences and their inverses so that the two pools remain nonadjacent in every later completion. No protected codegree has changed, and
\[
 \Diff(P_1\cup P_2)\cap D=\varnothing.
\]

With the two pools prepared, apply Lemma~\ref{lem:star} to add $L$ distinct vertices $C_1$, each adjacent to $P_1$ and nonadjacent to $P_2\cup\{1\}$. Next let $C_3$ consist of $1$ and $L-1$ new vertices, each adjacent to $P_1\cup P_2$ and nonadjacent to $C_1$. Exclude previously used marker vertices at each choice, while protecting the same fixed $D$ in every insertion. The four disjoint blocks
\[
 C_1,\quad P_1,\quad C_3,\quad P_2
\]
have the cross-pattern in Figure~\ref{fig:robust}, so Lemma~\ref{lem:p4} proves robustness independently of within-block edges. The edge $\{1,a\}$ is in this marker. We also need the marker construction to close under the new selected differences: every newly selected pool element occurs on its edge to $1$, and every star element $tp$ occurs on the edge from $t$ to $p^{-1}$ in the same symmetric pool; inverse differences use the reverse orientations. No new selected difference therefore creates a further marking obligation, and all later assignments retain the prescribed cross-edges and cross-nonedges.
\end{proof}

\subsection{The countable theorem}

\begin{theorem}[Countable realization with finite prescriptions]\label{thm:countable}
For every countably infinite group $H$, every finite disjoint symmetric prescription $A_0,R_0\subseteq H\setminus\{1\}$, and every fixed $m\ge0$, there is a closed Cantor family of symmetric sets $U_b$ satisfying that prescription. Both their Cayley graphs and their complements realize $\Xi(H)$, have diameter two, and contain every finite vertex set in a finite induced $m$-robust subgraph. The graphs are pairwise nonisomorphic, as are their complements. Every finite vertex set has infinitely many common neighbors in each complementary graph. Applied with $A_0$ and $R_0$ interchanged, the sets $S_b=(H\setminus\{1\})\setminus U_b$ form a closed Cantor family satisfying the prescription $A_0,R_0$, with $\Cay(H,S_b)$ the complement of $\Cay(H,U_b)$; so every finite vertex set has infinitely many common neighbors in each $\Cay(H,S_b)$.
\end{theorem}

The proof builds $U_b$ by finite marked insertions, freezing the common-neighbor count of the $i$-th primary pair at a value $M_i$ that the bit $b_i$ raises by one further insertion or leaves unchanged; the strictly increasing sequence $(M_i)$ then separates the graphs.

\begin{proof}
\emph{A single finite operation.} Maintain finite symmetric selected and forbidden sets, written $A_n,R_n$ after $n$ finite operations, so that $A_0,R_0$ is the prescription; they stay disjoint and only grow. Maintain also a finite symmetric set $D$ consisting only of primary pairs whose counts have already been frozen, and the marker invariant of Lemma~\ref{lem:markerprescription}: every selected difference is covered by a protected finite induced $m$-robust marker.

For an uncoded primary target or a secondary target, its seed $P$ satisfies
\[
\Diff(P)\cap D=\varnothing,
\]
since $\Diff(P)$ lies in the target pair and the secondary pairs (Lemma~\ref{lem:countable}). Lemma~\ref{lem:seedinsert} therefore inserts a fresh symmetric block $Y$ while preserving every count indexed by $D$. For each $w\in Y$, also $w^2\notin D$: otherwise $(w^{-1},w)$ would be a prohibited new--new representation of a member of $D$. Apply Lemma~\ref{lem:markerprescription} to each member of this finite $Y$, keeping $D$ fixed throughout. Its closure property covers every difference newly selected by the marker itself, so this creates no recursive marking obligation.

Call the resulting finite operation a \emph{marked insertion}. It preserves all frozen counts and all earlier markers, and increases the target count by at least one. Its size need not be bounded. Initially $D=\varnothing$; applying Lemma~\ref{lem:markerprescription} to the members of $A_0$ establishes the marker invariant.

\emph{Coding without windows.} Enumerate the primary inverse pairs as $\invp{a_i}$, $i\in\N$. Fix enumerations of all choices used in the construction and always take the first admissible choice. For a binary sequence $b$, at stage $i$:
\begin{enumerate}[label=(\arabic*),nosep]
\item Without consulting $b_i$, make marked insertions for $a_i$ until its current count $N_i$ is positive and greater than every previously frozen count.
\item If $b_i=0$, freeze $M_i=N_i$.
\item If $b_i=1$, make one further complete marked insertion for $a_i$, and freeze its resulting count $M_i>N_i$.
\item Add $\invp{a_i}$ to $D$.
\end{enumerate}
All these operations terminate because each insertion raises the target count and the threshold is finite. Every branch produces
\[
0<M_0<M_1<M_2<\cdots,\qquad \nu_{U_b}(a_i)=M_i.
\]
After coding at each stage, interleave the following three finite tasks.
\begin{itemize}[nosep]
\item Saturate secondary pairs using a fixed schedule in which every secondary pair occurs infinitely often. Index the secondary list by an initial segment of $\N$, allowing the list to be finite or empty.
\item Reserve complement witnesses. Schedule every finite $F\subseteq H$ infinitely often. Choose a fresh $z\notin F$ with $zf^{-1}\notin A_n$ for all $f\in F$, and put all these inverse pairs into $R_{n+1}$. Only finitely many $z$ are excluded. Excluding the finitely many witnesses used earlier makes the witnesses distinct.
\item Enumerate all nonidentity inverse pairs $\invp{g_i}$. At the end of stage $i$, forbid $\invp{g_i}$ if it has not yet been selected. This decides that connection-set coordinate permanently.
\end{itemize}
The extra forbidden pairs do not obstruct later work: both seed insertion and protected marking allow arbitrary finite forbidden data.

\emph{Verification.} Let $U_b=\bigcup_nA_n$; it contains $A_0$ and is disjoint from $R_0$. Every primary count remains exactly its frozen finite value, since marked insertions preserve the counts indexed by $D$, and every secondary count becomes infinite, since each secondary pair receives infinitely many marked insertions. Every representation in the union already occurs at a finite stage, so no additional primary representation appears only at the limit. Lemma~\ref{lem:recognition} gives
\[
\Aut(\Cay(H,U_b))=\Xi(H),\qquad \nu_{U_b}(h)>0\quad(h\ne1).
\]
The marked edges consequently connect $H$. By Lemma~\ref{lem:robustunion}, every finite set of vertices is contained in a finite induced $m$-robust subgraph, and this assertion also holds in the complement. The reserved witnesses give every finite vertex set infinitely many common neighbors in the complement.

To prove nonisomorphism, use the increasing spectrum of finite common-neighbor counts on distinct vertex pairs. It is exactly $(M_0,M_1,\ldots)$. If two binary sequences first differ at $i$, their construction histories before that choice agree, including $N_i$ and all earlier frozen counts. One spectrum has its $i$-th entry $N_i$, while the other has a larger $i$-th entry. The value $N_i$ cannot appear later in the second spectrum because that spectrum is strictly increasing, and cannot appear earlier because the earlier entries coincide. The graphs are therefore nonisomorphic. The same follows for their complements.

The decision of coordinate $\invp{g_i}$ depends only on the first $i+1$ input bits. Thus $b\mapsto U_b$ is continuous. It is injective, so compactness of $\{0,1\}^{\N}$ and the Hausdorff property of the connection-set space $\mathcal X_H$ make it an embedding with closed image.

For diameter, the uncomplemented graphs have positive codegree at every nonidentity difference; the complemented graphs have the stronger reserved-witness property. Both connection sets are nonempty: the first follows from positive codegrees, and the second from the reserved witnesses. Hence neither graph is complete, and the diameter bounds are exactly two.

For the sets $S_b$, run the construction with the selected and forbidden prescriptions reversed and set $S_b=(H\setminus\{1\})\setminus U_b$. Complementation is a homeomorphism and preserves both automorphism groups and isomorphism distinctions.
\end{proof}

\section{Hierarchy recovery and finite packets}\label{sec:recovery}
We begin with the reconstruction principle for arbitrary graphs. Group multiplication appears only later, when the group constructions are used to realize its hypotheses.

\begin{definition}[Continuous uniformly branching hierarchy]
Let $X$ be a nonempty set. A hierarchy is a family $(E_\alpha)_{\alpha\le\theta}$ of equivalence relations with:
\begin{enumerate}[label=(\roman*),nosep]
\item $E_\alpha\subseteq E_\beta$ for $\alpha<\beta$, and $E_\theta=X\times X$;
\item $E_\lambda=\bigcup_{\alpha<\lambda}E_\alpha$ at every nonzero limit;
\item every $E_{\alpha+1}$-class contains at least two $E_\alpha$-classes.
\end{enumerate}
For $x,y$ in different $E_0$-classes, let $r(x,y)$ be the unique $\alpha$ with $(x,y)\in E_{\alpha+1}\setminus E_\alpha$. Choose $\eps_\alpha\in\{0,1\}$ for $\alpha<\theta$, with $\eps_{\alpha+1}=1-\eps_\alpha$ whenever both indices occur.
\end{definition}
Uniform branching is needed for the witnesses used below; merely requiring $E_\alpha\ne E_{\alpha+1}$ somewhere does not provide enough local branching.

\begin{definition}[Layer errors]
For the integer $m\ge0$ of Definition~\ref{def:robustgraph} and a graph $\Gamma$ on $X$, put
\[
 \mathcal L_\alpha=\{\{x,y\}:(x,y)\in E_{\alpha+1}\setminus E_\alpha,
                 \ \mathbf1_{E(\Gamma)}(\{x,y\})\ne\eps_\alpha\}.
\]
The error condition is $\maxdeg(X,\mathcal L_\alpha)\le m$ for each $\alpha$.
When an unordered edge set is intersected with $E_{\alpha+1}\setminus E_\alpha$,
we mean its pairs whose two endpoints are in that equivalence-relation layer. This is a bound per vertex in the \emph{whole layer}, rather than separately in each pair of initial classes. Total error degree over all levels need not be bounded.
\end{definition}

\begin{lemma}[Twin classes of the baseline]\label{lem:twins}
Let $(E_\alpha)_{\alpha\le\theta}$ be a continuous uniformly branching hierarchy on $X$ with successor-alternating bits $(\eps_\alpha)_{\alpha<\theta}$, and let $\alpha<\theta$. The \emph{baseline graph} on $X/E_\alpha$ joins distinct classes $P,Q$ exactly when $\eps_{r(x,y)}=1$ for $x\in P$ and $y\in Q$; this does not depend on the choice of $x$ and $y$. For a graph $J$, let $\operatorname{Tw}(J)$ be equality together with the twin relation: distinct vertices have the same adjacency to every third vertex, ignoring their mutual adjacency. In the baseline graph on $X/E_\alpha$, the twin classes are exactly $E_{\alpha+1}/E_\alpha$; that is, $\operatorname{Tw}$ of this graph is the relation induced by $E_{\alpha+1}$ on $X/E_\alpha$.
\end{lemma}

\begin{proof}
If $(x,x')\in E_\alpha$ and $r(x,y)=\beta$ with $(x,y)\notin E_\alpha$, then $\beta\ge\alpha$, and nesting gives $(x',y)\in E_{\beta+1}\setminus E_\beta$; so the rank of two distinct $E_\alpha$-classes is well defined, and so is the baseline graph. Classes in one parent have identical adjacency to every third class: their ranks to a third class are both $\alpha$ when it lies in the same parent, and otherwise agree by the same nesting argument applied to $E_{\alpha+1}$. For classes in different parents, let their mutual rank be $\beta>\alpha$. Choose $\gamma=\alpha$ if $\eps_\alpha\ne\eps_\beta$, and $\gamma=\alpha+1$ otherwise. In the latter case successor alternation gives $\gamma<\beta$ and $\eps_\gamma\ne\eps_\beta$. Uniform branching supplies a third class at rank $\gamma$ from the first and rank $\beta$ from the second, distinguishing them. Indeed, for $x$ in the first class, the $E_{\gamma+1}$-class of $x$ contains a second $E_\gamma$-class, so some $z$ has $r(x,z)=\gamma$; since $\gamma+1\le\beta$, nesting gives $r(y,z)=\beta$ for $y$ in the second class. As $\gamma,\beta\ge\alpha$, the class of $z$ is a third class, and its adjacencies to the two classes are $\eps_\gamma\ne\eps_\beta$.
\end{proof}

\begin{theorem}[Intrinsic recovery without a group]\label{thm:recovery}
Fix a nonnegative integer $m$. Suppose $\Gamma$ has a hierarchy as above, every $E_0$-class has more than $2m$ vertices, and:
\begin{enumerate}[label=(\roman*),nosep]
\item $\maxdeg(X,\mathcal L_\alpha)\le m$ for every $\alpha$;
\item inside every $E_0$-class, the edges belonging to finite induced $m$-robust subgraphs of $\Gamma$ form a connected spanning graph.
\end{enumerate}
Then adjacency alone, with the fixed parameter $m$, determines $E_0$, which is the relation $\sim_m$ of Lemma~\ref{lem:robustunion}, every $E_\alpha$, the bits $\eps_\alpha$, the ordinal height $\theta$, and every error relation $\mathcal L_\alpha$. The recovery is canonical and ordinal-indexed: if $f:\Gamma\to\Gamma'$ is an isomorphism between two such graphs (with the same parameter $m$), then $(f\times f)[E_\alpha]=E'_\alpha$ at every ordinal $\alpha$. In particular, their heights and corresponding baseline bits agree.
\end{theorem}
Here recovery refers to the transfinite procedure constructed in the proof from
$(\Gamma,m)$ itself; the hierarchy is not supplied as additional input.

The plan: no finite induced $m$-robust subgraph crosses an $E_0$-class, so $E_0$ is $\sim_m$; a majority test then reads off the baseline $\Delta_0$, and twin relations of its quotients (Lemma~\ref{lem:twins}) give each $E_{\alpha+1}$, with unions at limits.
\begin{proof}
First, let $V$ be the vertex set of a finite induced $m$-robust subgraph. We must show that such a marker cannot cross initial classes, since its components will be used to recover those classes. Fix one of its vertices $x_0$ and take the largest rank $\beta$ among its exterior comparisons with $x_0$. Then all its vertices lie in one $E_{\beta+1}$-class, while at least two $E_\beta$-classes occur. Consider the cut of the finite vertex set $V$ into $V\cap Y$ and $V\setminus Y$, where $Y$ is one $E_\beta$-class meeting $V$. Every crossing pair has rank $\beta$, so each vertex has at most $m$ deviations from the constant bit $\eps_\beta$. If that bit is zero, every crossing degree is at most $m$; if it is one, every crossing nondegree is at most $m$. Either alternative contradicts robustness, proving the required containment.

By Lemma~\ref{lem:robustunion}(ii), the relation $\sim_m$, defined from adjacency by finite witnesses, has as its classes the components of $(X,M_m)$, where $M_m$ is the set of robust-marker edges of $\Gamma$. The containment just proved puts every such component inside one $E_0$-class, and condition (ii) makes every $E_0$-class connected in $(X,M_m)$. Hence the recovered initial relation $\widehat E_0={\sim_m}$ equals $E_0$; This completes the first recovery step.

For distinct initial classes $B,C$ and $x\in B$, the rank to all points of $C$
is constant. According to the bit, there are at most $m$ neighbors in $C$ or
at most $m$ nonneighbors, and $|C|>2m$ makes these alternatives disjoint. This separation lets us recover the baseline $\Delta_0$ on
$Q_0=X/\widehat E_0$ by the majority test
\[
 \{B,C\}\in E(\Delta_0)
 \quad\Longleftrightarrow\quad
 |C\setminus N_\Gamma(x)|\le m\qquad\text{for any }x\in B,\ \text{where }B\ne C,
\]
where $N_\Gamma(x)$ is the neighbor set of $x$. Its value is independent of
$x\in B$ and symmetric in $B,C$, because it is $\eps_{r(x,y)}$ for
any $y\in C$; thus $\Delta_0$ is the baseline graph on $X/E_0$ in the sense of Lemma~\ref{lem:twins}. Every higher level is recovered from this baseline, not from $\Gamma$, whose layers carry errors.

The higher levels follow one recurrence: twin quotients of the recovered baseline at successors and unions at limits, with $\operatorname{Tw}$ as in Lemma~\ref{lem:twins}. If $\widehat E_\alpha\supseteq\widehat E_0$ is already
recovered, let $q_\alpha:X\to X/\widehat E_\alpha$ be the quotient map, and let
$\widehat E_\alpha^0$ be its induced relation on $Q_0$:
\[
 B\,\widehat E_\alpha^0\,C
 \quad\Longleftrightarrow\quad
 (x,y)\in\widehat E_\alpha\qquad\text{for all }x\in B,\ y\in C.
\]
Starting with $\widehat E_0$ and $\Delta_0$ above, set
\[
\begin{aligned}
 \Delta_\alpha&=\Delta_0/\widehat E_\alpha^0,
       &&\text{on }X/\widehat E_\alpha,\\
 \widehat E_{\alpha+1}&=(q_\alpha\times q_\alpha)^{-1}
                   \bigl(\operatorname{Tw}(\Delta_\alpha)\bigr),
       &&\text{at successors},\\
 \widehat E_\lambda&=\bigcup_{\alpha<\lambda}\widehat E_\alpha,
       &&\text{at nonzero limits},\\
 \widehat\theta&=\min\{\alpha:\widehat E_\alpha=X\times X\}.
\end{aligned}
\]
Loops are discarded in the quotient graph, and the successor relation is pulled
back to $X$ rather than left on a changing quotient vertex set. We prove by transfinite induction on $\alpha\le\theta$ that every displayed quotient is well defined and that $\widehat E_\alpha=E_\alpha$; the case $\alpha=0$ is the first recovery step. If $\widehat E_\alpha=E_\alpha$, then adjacency in $\Delta_0$ between $E_0$-classes in distinct $E_\alpha$-classes depends only on those $E_\alpha$-classes, since nesting makes their pair rank constant. Hence $\Delta_\alpha$ is well defined and is the baseline graph on $X/E_\alpha$, and for $\alpha<\theta$ Lemma~\ref{lem:twins} gives $\operatorname{Tw}(\Delta_\alpha)=E_{\alpha+1}/E_\alpha$, so $\widehat E_{\alpha+1}=E_{\alpha+1}$. At a nonzero limit $\lambda\le\theta$, the displayed union and continuity give $\widehat E_\lambda=E_\lambda$. This completes the transfinite induction, giving
$\widehat E_\alpha=E_\alpha$ at every stage; strict branching prevents
premature stabilization, so $\widehat\theta=\theta$. With the hierarchy in hand, for each $\alpha<\theta$ choose two distinct $E_\alpha$-classes inside one $E_{\alpha+1}$-class; uniform branching ensures that such classes exist. Their adjacency in the recovered baseline quotient $\Delta_\alpha$ is $\eps_\alpha$. The displayed definition of $\mathcal L_\alpha$, using the original adjacency of $\Gamma$, then recovers the error relation. Finally, let $f:\Gamma\to\Gamma'$ be an isomorphism. It carries finite induced $m$-robust subgraphs onto finite induced $m$-robust subgraphs, so $(f\times f)[\widehat E_0]=\widehat E'_0$, and it preserves the majority test, so it induces an isomorphism $\Delta_0\to\Delta'_0$. Since isomorphisms preserve twin relations and the recurrence takes unions at limits, induction on $\alpha$ gives $(f\times f)[\widehat E_\alpha]=\widehat E'_\alpha$ at every ordinal $\alpha$, with the same index on both sides. As $\widehat E_\alpha=E_\alpha$ and $\widehat E'_\alpha=E'_\alpha$, this is $(f\times f)[E_\alpha]=E'_\alpha$ at every ordinal $\alpha$; hence the heights agree, and $f$ induces isomorphisms $\Delta_\alpha\to\Delta'_\alpha$ carrying two classes inside one $E_{\alpha+1}$-class to two classes inside one $E'_{\alpha+1}$-class, so the bits read off above agree. This proves the asserted canonicity.
\end{proof}

\begin{remark}[Scope of the hypotheses]
No Cayley action, group operation, normal subgroup, transitivity, countability, or equality of class sizes was used. What remains essential is the continuous ordinal hierarchy, uniform branching, alternating bits, a fixed finite layer-error budget, and connected robust markers. These hypotheses restrict the class of hierarchies and perturbations to which the theorem applies.
\end{remark}

\begin{proposition}[Explicit non-Cayley examples at every ordinal height]\label{prop:ordinalexamples}
For every ordinal $\theta$ and every successor-alternating bit sequence $(\eps_\alpha)_{\alpha<\theta}$, there is a graph with an intrinsically recoverable hierarchy of height exactly $\theta$ and those bits. The graph is not vertex-transitive. If $\theta$ is infinite, its vertex set has cardinality $|\theta|$.
\end{proposition}
\begin{proof}
Use the four-block path $J$ with three vertices in each block, no edges within a block, all edges between consecutive blocks, and no other edges. Lemma~\ref{lem:p4}, with $m=1$, shows that $J$ is 1-robust, providing the marker needed in each initial class.

Put $V=\F_2^{(\theta)}$ and $X=V\times V(J)$. Between distinct such sets with coordinates $v,w$, give every pair adjacency $\eps_{\max\operatorname{supp}(v+w)}$. Define
\[
 (v,i)\ E_\alpha\ (w,j)\iff
 \operatorname{supp}(v+w)\subseteq\alpha.
\]
The resulting hierarchy is continuous, each successor class merges exactly two previous classes, and its layer errors are empty. Its initial classes have twelve vertices and connected robust-marker graphs, so Theorem~\ref{thm:recovery} recovers its height and bits. To see that the example is not vertex-transitive, look inside an intrinsically recovered initial class: vertices in the end blocks of $J$ have internal degree three, while vertices in its middle blocks have internal degree six. An automorphism cannot exchange those types, proving nontransitivity. The cardinality assertion follows from finite support.
\end{proof}
The construction therefore gives an explicit application outside the class of Cayley graphs.

\subsection{A quantitative margin for further edits}
\begin{lemma}[Robustness loses at most the edit degree]\label{lem:margin}
If a finite graph $F$ is $(q+d)$-robust and the symmetric difference between its edge set and that of $F'$ has maximum degree at most $d$, then $F'$ is $q$-robust.
\end{lemma}
\begin{proof}
Across any cut, editing changes both the crossing-neighbor count and the crossing-nonneighbor count of a fixed vertex by at most $d$. Applying this loss bound to the two witnesses in the robustness condition leaves the required witness of each type in place.
\end{proof}

\begin{corollary}[Hierarchy recovery under layerwise bounded edits]\label{cor:edits}
Let $m,d$ be nonnegative integers. Suppose the hierarchy and graph have layer-error budget $m$, every initial class has more than $2(m+d)$ vertices, and its internal marker-edge relation is connected using $(m+2d)$-robust finite patterns. Let $\mathcal T$ be a set of unordered pairs to be toggled. Assume that $\mathcal T$ has maximum degree at most $d$ inside each initial class and, separately, that
\[
\maxdeg\bigl(X,\mathcal T\cap(E_{\alpha+1}\setminus E_\alpha)\bigr)\le d
\qquad\text{for all }\alpha<\theta,
\]
where equivalence-relation differences are interpreted as sets of unordered distinct pairs. The edited graph intrinsically recovers the original hierarchy with threshold $m+d$, including its original baseline bits. The recovered error relations are those of the edited graph.
\end{corollary}
\begin{proof}
We check exactly the two hypotheses needed to reapply the recovery theorem. In each exterior layer the new error relation is the symmetric difference of two graphs of maximum degrees at most $m$ and $d$, so its degree is at most $m+d$. Inside each initial class, an old marker vertex set induces an $(m+d)$-robust graph after editing, by Lemma~\ref{lem:margin}, and is therefore connected. Even when an old marker edge is deleted, its endpoints remain connected inside that marker vertex set. Since the old marker-edge relation spanned and connected the class, the surviving marker connections still span and connect it. Theorem~\ref{thm:recovery} now applies to the same hierarchy and bits with the new threshold.
\end{proof}
The stability condition is layerwise: the total edit degree over all layers may be infinite. A globally bounded-degree edit is only a special case, not an additional requirement.

\begin{remark}[Not resilience of the GRR]
The corollary preserves the hierarchy, not its original packets or its full automorphism group. In a one-layer Cayley construction, filling in all finitely many missing packet directions is a bounded-degree edit. The result may be the complete join of two core graphs, allowing their translation actions to vary independently. Losing the pinning information can enlarge the automorphism group even while the core classes remain recoverable. Nor does this give recovery after an unrestricted positive fraction of adversarial edits.
\end{remark}

\begin{proposition}[Complexity of robustness]\label{prop:robust-conp}
For each fixed integer $m\ge 1$, deciding whether a finite graph is $m$-robust is coNP-complete. For $m=0$, it is decidable in polynomial time.
\end{proposition}

\begin{proof}Nonrobustness has a polynomially checkable certificate: either the graph has fewer than two vertices, or it or its complement has an $m$-cut. For hardness, Gomes--Sau \cite[Theorem~1]{GomesSau2019} prove that $m$-Cut is NP-hard on $(2m+2)$-regular graphs. If such a graph $J$ has $n>6m+4$ vertices, its complement has no $m$-cut. Indeed, in such a cut $A\sqcup B$, a vertex in $A$ would have at least $|B|-m$ neighbors in $B$ in $J$. Since $J$ has degree $2m+2$, this gives $|B|\le 3m+2$; similarly, $|A|\le 3m+2$, contradicting $n>6m+4$. Hence $J$ is $m$-robust exactly when $J$ has no $m$-cut. For fixed $m$, the inputs with at most $6m+4$ vertices can be decided exhaustively and mapped to a fixed robust or nonrobust graph according as they have no $m$-cut or have one; Lemma~\ref{lem:p4} supplies the robust graph, and $K_2$ supplies the nonrobust graph. This gives coNP-hardness by a polynomial many-one reduction from the complement of that restricted $m$-Cut problem. For $m=0$, robustness is equivalent to connectedness of both the graph and its complement, with at least two vertices, and is decidable in polynomial time.
\end{proof}

\begin{remark}[Finite recognition and modules]\label{rem:finite-recognition}
In the perturbed graph the recovered classes need not be modules in the sense of modular decomposition: an exterior vertex may distinguish vertices in one initial class. For finite input and a fixed marker-size bound, direct enumeration gives a recognition procedure; without a marker-size bound, Proposition~\ref{prop:robust-conp} applies.
\end{remark}

\subsection{Finite packets relative to an exact subgroup action}\label{sec:packets}
Let $\mathcal E_H=\{F\in\Sym(G):F(hx)=hF(x)\text{ for all }h\in H,x\in G\}$.

\begin{lemma}[Sandwich equations]\label{lem:sandwich}
Let $H$ be an infinite nonabelian group. For distinct $x,y$ in an overgroup, there are infinitely many $h\in H$ such that $\invp{hx}\ne\invp{hy}$.
\end{lemma}
\begin{proof}
Since $hx\ne hy$, equality of the inverse pairs is equivalent to $hxh=y^{-1}$, so it is enough to rule out a cofinite solution set. Suppose that its solution set is cofinite. Fix a solution $r$ and put $h=kr$. With $z=rx$, the equation becomes $kzk=z$, which says that conjugation $\alpha(k)=z^{-1}kz$ agrees with inversion on a cofinite set $\Omega\subseteq H$. The next computation turns that large-set agreement into a restriction on every fixed group element. For fixed $a\in H$ and $k\in \Omega\cap a^{-1}\Omega$,
\[
 \alpha(a)=\alpha(ak)\alpha(k)^{-1}=k^{-1}a^{-1}k.
\]
The intersection is cofinite. A nonempty fiber of $k\mapsto k^{-1}a^{-1}k$ is a coset of $C_H(a)$, and a proper subgroup coset cannot be cofinite in an infinite group. Hence every $a$ is central in $H$, contradicting the hypothesis and proving the lemma.
\end{proof}

\begin{lemma}[Uniform relative packet]\label{lem:packet}
Suppose $H\le M\le G$, $a\notin M$, and $H$ is infinite and nonabelian. There is
\[
a\in W=W^{-1}\subseteq\cl{H,a}\setminus M,\qquad |W|\le76,
\]
such that every $F\in\mathcal E_H$ with $F(W)=W$ fixes $Ha$ pointwise.
\end{lemma}

The proof fixes a core $C\ni a$ of at most six points that every $F\in\mathcal E_H$ with $F(W)=W$ maps onto itself, and then, for each of the at most $35$ nonidentity ways such an $F$ can act on $C$, adds to $W$ one inverse pair whose forced image is kept outside $W$.
\begin{proof}
\emph{Initial set and candidates.} Choose distinct $1,h_1,h_2\in H$ and let
\[
 W_0=\{a,h_1a,h_2a\}\cup\{a,h_1a,h_2a\}^{-1}.
\]
Put
\[
 \mathcal D=\{xy^{-1}:x,y\in W_0,\ x\ne y,\ xy^{-1}\in H\}.
\]
To recognize the original core after later witness insertions, define the auxiliary $\mathcal D$-graph on any subset $X_0\subseteq G$ by
\[
x\sim_{\mathcal D} y\quad\Longleftrightarrow\quad x,y\in X_0,\ x\ne y,\ xy^{-1}\in \mathcal D.
\] In the auxiliary $\mathcal D$-graph on $W_0$, the nonempty intersections with right $H$-cosets are complete components. Let $C$ be the union of components of size at least three. Then $a\in C$, $|C|\le6$, and $C$ meets at most two right $H$-cosets.

\emph{Witness insertions.} There are at most $6^2=36$ bijections $\phi:C\to C$ extending to an $H$-equivariant permutation $\bar\phi$ of $HC$: the image of one representative of each occupied coset determines the extension. Call these bijections \emph{candidates}. Start with $W=W_0$ and a symmetric reserved set $R=\varnothing$, and for each nonidentity candidate choose $c\in C$ moved by it. Applying
Lemma~\ref{lem:sandwich} to the distinct elements $c$ and $\phi(c)$ gives
infinitely many $h\in H$ with $\invp{hc}\ne\invp{h\phi(c)}$, and hence infinitely many
$x\in Hc$ with $\bar\phi(x)\notin\invp{x}$, providing a supply of possible witnesses.
Require in addition
\[
 \invp{x}\cap(W\cup R)=\varnothing,\qquad \invp{x}\cap \mathcal D W=\varnothing,
 \qquad \bar\phi(x)\notin W.
\]
Only finitely many $x$ are excluded: on $Hc$, both $x\mapsto x$ and $x\mapsto\bar\phi(x)$ are injective, and each prescribed inverse pair has at most two preimages. An admissible $x$ therefore remains; add $\invp{x}$ to $W$ and $\invp{\bar\phi(x)}$ to $R$ so this candidate can no longer preserve the final packet.
The symmetric sets $W,R$ remain disjoint. Because $W=W^{-1}$, the condition $\bar\phi(x)\notin W$ also excludes its inverse from the old $W$, while $\bar\phi(x)\notin\invp{x}$ excludes both from the newly added pair. Finally, $\invp{x}\cap R=\varnothing$ protects every reservation made for an earlier candidate.

\emph{Conclusion.} The new pair is auxiliary-separated from the old $W$, so its component has size at most two. Consequently $C$ remains exactly the union of the auxiliary components of size at least three. There are at most 35 nonidentity candidates, whence $|W|\le6+2\cdot35=76$ and $|R|\le70$. All selected and reserved points lie in $\cl{H,a}\setminus M$.

An $H$-equivariant permutation preserves $\mathcal D$-adjacency: $x=dy$ with $d\in \mathcal D\subseteq H$ implies $F(x)=dF(y)$, and the inverse permutation gives the converse. If it preserves the final $W$, it therefore preserves its auxiliary $\mathcal D$-graph and hence $C$, so its restriction is one of the candidates. A nonidentity candidate has a retained witness $x=hc\in W$ whose required image $h\phi(c)=\bar\phi(x)$ is still reserved outside $W$, a contradiction. The restriction must therefore be the identity, which fixes $C$ and $Ha$ as required.
\end{proof}

\begin{lemma}[Boolean packet]\label{lem:booleanpacket}
If $G$ is Boolean, $H\le M\le G$, $|H|\ge4$, and $a\notin M$, choose distinct nonidentity $p,q\in H$ (available because $|H|\ge4$) and put $W=\{1,p,q\}a$. This symmetric three-point set lies outside $M$, contains $a$, and has the conclusion of Lemma~\ref{lem:packet}.
\end{lemma}
\begin{proof}
If $F$ commutes with the $H$-action and preserves $W$, then $F(a)=ha$ for some $h\in H$ and $\{1,p,q\}h=\{1,p,q\}$. The second equality reduces the problem to the translation stabilizer of this three-point set. A nonidentity Boolean translation acts by two-cycles and cannot preserve an odd-cardinality set, so $h=1$ and the required coset is fixed pointwise. Finally, every element of a Boolean group is its own inverse, so $W=W^{-1}$; and $ka\in M$ with $k\in H\le M$ would force $a\in M$, so $W\cap M=\varnothing$.
\end{proof}

\begin{table}[tb]
\centering\small
\begin{tabular}{@{}p{.22\textwidth}p{.7\textwidth}@{}}
\toprule
Parameter & Role \\
\midrule
$m=76$ & Uniform bound from Lemma~\ref{lem:packet} on a nonabelian relative packet and on each layer's error degree. \\
$L=2m+1$ & Minimum size of a block in the prescribed four-block marker. \\
$18$ & Dimension of the Boolean decoder. \\
$96$ & Color capacity of the Boolean decoder (tags $\Pi_1,\ldots,\Pi_{96}$); Theorem~\ref{thm:finitecolor} supplies at most $17$ colors, and the binary refinement in the proof of Proposition~\ref{prop:exceptmany} uses at most $34\le96$. \\
$33/67$ & Largest odd tag sizes for the two working palettes in $C_2^{18}$ (the $17$ colors of Theorem~\ref{thm:finitecolor} and the $34$ of its binary refinement). \\
\bottomrule
\end{tabular}
\caption{Construction constants. None is an asserted optimal bound, except where a separate optimality statement is proved.}
\label{tab:constants}
\end{table}

\section{Regular-case construction and maximal abundance}\label{sec:main}\subsection{Alternating-layer construction}\label{sec:layers}
The regular-case construction is organized along a continuous increasing chain with strict successor inclusions
\[
 H=G_0<G_1<\cdots<G_\theta=G.
\]
The set $G_{\alpha+1}\setminus G_\alpha$ is called \emph{layer $\alpha$}, for $\alpha<\theta$.
For $g\notin H$, let $r(g)=\alpha$ when $g\in G_{\alpha+1}\setminus G_\alpha$; continuity ensures that this is defined. Choose bits $\eps_\alpha$ for $\alpha<\theta$ satisfying
\begin{equation}\label{eq:flip}
 \eps_{\alpha+1}=1-\eps_\alpha\quad\text{for all }\alpha\text{ with }\alpha+1<\theta.
\end{equation}
At $0$ and limit indices the bit is free. Let
\[
 W_\alpha=W_\alpha^{-1}\subset G_{\alpha+1}\setminus G_\alpha,
 \qquad |W_\alpha|\le m.
\]
Given a core connection set $U\subset H\setminus\{1\}$, define $S$ by
\begin{equation}\label{eq:S}
 S\cap H=U,\qquad
 \mathbf1_S(g)=\eps_{r(g)}\mathbin\oplus\mathbf1_{W_{r(g)}}(g)
 \quad\text{for all }g\notin H.
\end{equation}
Because every $G_\alpha$ is a subgroup, $r(g^{-1})=r(g)$.
Together with $U=U^{-1}$ and $W_\alpha=W_\alpha^{-1}$, this shows that
$S=S^{-1}$. Each layer is therefore almost empty or almost complete, with its packet
as the finite exceptional set. Figure~\ref{fig:layers} records this rule and the independent choices available at limit stages.
\begin{figure}[!htbp]
\centering
\begin{tikzpicture}[x=1cm,y=1cm,every node/.style={font=\small}]
  \path[use as bounding box] (-3.55,-.30) rectangle (9.90,6.45);
  \draw[fill=white] (0,0) rectangle (5.5,.80);
  \node at (2.75,.48) {core: $S\cap H=U$};
  \node[font=\footnotesize] at (2.75,.17) {$\Cay(H,U)$ is specified separately};
  \node[anchor=east] at (-.22,.40) {$H=G_0$};
  \foreach \lo/\hi/\a/\bit in {.80/1.65/0/1,1.65/2.50/1/0,2.50/3.35/2/1}{
    \ifnum\bit=1
      \draw[fill=black!16] (0,\lo) rectangle (5.5,\hi);
      \foreach \px in {1.0,2.65,4.45}{\draw[fill=white] (\px,{(\lo+\hi)/2}) circle (2.2pt);}
    \else
      \draw[fill=white] (0,\lo) rectangle (5.5,\hi);
      \foreach \px in {1.0,2.65,4.45}{\fill (\px,{(\lo+\hi)/2}) circle (2.2pt);}
    \fi
    \node[anchor=east] at (-.22,{(\lo+\hi)/2}) {layer $\a$: $\eps_\a=\bit$};
    \node[anchor=west] at (5.72,{(\lo+\hi)/2}) {$W_\a$ (exceptions)};
  }
  \draw[densely dotted] (0,3.35)--(0,4.05);
  \draw[densely dotted] (5.5,3.35)--(5.5,4.05);
  \node at (2.75,3.70) {$\vdots$};
  \node[anchor=west,font=\footnotesize] at (5.72,3.70) {omitted stages};
  \draw[fill=black!16] (0,4.05) rectangle (5.5,4.90);
  \foreach \px in {1.0,2.65,4.45}{\draw[fill=white] (\px,4.475) circle (2.2pt);}
  \node[anchor=east] at (-.22,4.475) {layer $\lambda$: $\eps_\lambda=1$};
  \node[anchor=west,align=left,font=\footnotesize] at (5.72,4.475) {free limit bit;\\shown here as $1$};
  \draw[fill=white] (0,4.90) rectangle (5.5,5.75);
  \foreach \px in {1.0,2.65,4.45}{\fill (\px,5.325) circle (2.2pt);}
  \node[anchor=east] at (-.22,5.325) {layer $\lambda+1$: $0$};
  \node[anchor=west] at (5.72,5.325) {$W_{\lambda+1}$ (exceptions)};
  \node[anchor=west,font=\footnotesize] at (5.72,5.96) {successor bit flips};
  \node[font=\footnotesize] at (2.75,6.15) {each layer $\alpha$ is $G_{\alpha+1}\setminus G_\alpha$};
\end{tikzpicture}
\caption{The connection set in \eqref{eq:S}. A gray layer selects every element except its packet (open circles); a white layer selects only its packet (filled circles). The dots represent $W_\alpha=W_\alpha^{-1}$, with $|W_\alpha|\le m$ and $a_\alpha\in W_\alpha$. Each subgroup is the core together with all preceding layers. At the displayed limit $\lambda$, the free choice $\eps_\lambda=1$ is shown; choosing $0$ reverses the next finite alternating segment. Band heights and the number of dots are schematic; omitted stages include $G_\lambda=\bigcup_{\alpha<\lambda}G_\alpha$.}
\label{fig:layers}
\end{figure}
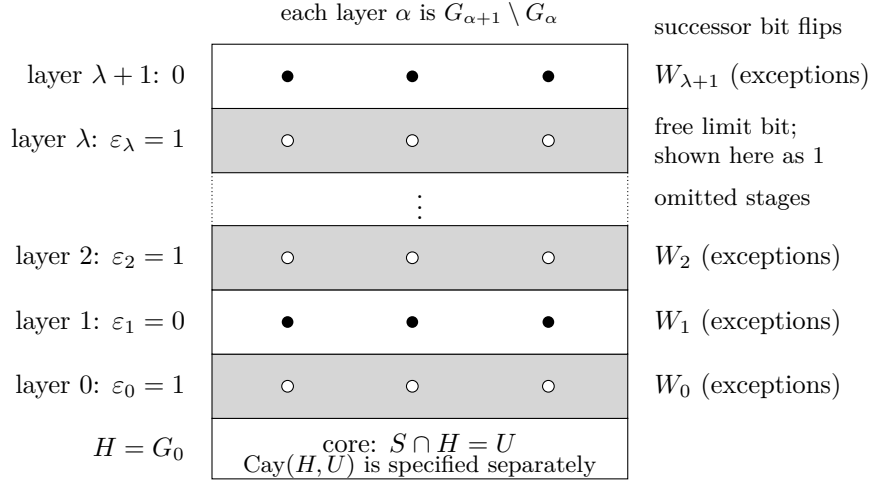

\begin{corollary}[Cayley specialization of hierarchy recovery]\label{cor:cayleyrecovery}
Suppose $|H|>2m$ and the robust-marker edge relation inside $\Cay(H,U)$ is connected. The graph defined by \eqref{eq:S} intrinsically determines every right $G_\alpha$-coset partition and every baseline bit. If an automorphism fixes $1$, it preserves each layer and each individual packet $W_\alpha$. If $S'$ is defined by \eqref{eq:S} from the same $H$, $U$, $(G_\alpha)$ and $(W_\alpha)$ and another successor-alternating bit sequence $(\eps'_\alpha)$, then every isomorphism $\Cay(G,S)\to\Cay(G,S')$ maps the partition of $G$ into right $G_\alpha$-cosets onto itself for every $\alpha\le\theta$, and if such an isomorphism exists, then $\eps'_\alpha=\eps_\alpha$ for every $\alpha<\theta$.
\end{corollary}
\begin{proof}
On the original set $G$, put $xE_\alpha y$ when $yx^{-1}\in G_\alpha$. To invoke the recovery theorem, we verify its hierarchy and error hypotheses for these right-coset equivalences. They are continuous, and strict successor inclusions imply branching inside every parent class. For a fixed vertex $x$, a discrepancy in layer $\alpha$ occurs only at vertices $wx$, with $w\in W_\alpha$; hence the error degree in the entire layer is at most $m$. Each initial class is a right coset $Hx$, so it has $|H|>2m$ points. Since $S\cap H=U$ by \eqref{eq:S}, the graph induced on $H$ is $\Cay(H,U)$, and right translation by $x$, an automorphism of $\Cay(G,S)$, carries it onto the graph induced on $Hx$. Hence every initial class carries the supplied connected robust-marker relation, so Theorem~\ref{thm:recovery} applies.

By the isomorphism clause of Theorem~\ref{thm:recovery}, an automorphism fixing $1$ maps the right $G_\alpha$-coset of $1$, which is $G_\alpha$, onto itself, so it preserves every $G_\alpha$ and its successor layer. The last step is to recover the finite exception set itself: in a zero layer the packet is the set of neighbors of $1$, while in a one layer it is the set of nonneighbors of $1$. Each packet is consequently individually invariant.

For the last assertion, the verification above applies verbatim to $S'$, whose hierarchy is the same family $(E_\alpha)$ of right-coset equivalences. For an isomorphism $f:\Cay(G,S)\to\Cay(G,S')$, Theorem~\ref{thm:recovery} gives $(f\times f)[E_\alpha]=E_\alpha$ at every ordinal $\alpha\le\theta$, which says that $f$ maps the partition into right $G_\alpha$-cosets onto itself, and it gives $\eps'_\alpha=\eps_\alpha$ for every $\alpha<\theta$.
\end{proof}
\subsection{Exhaustive choice of a core and exact translation recovery}
To choose a countable core in the nonabelian regular case, we need the following finite-witness reduction.
\begin{lemma}\label{lem:witness}
A nonabelian, nongeneralized-dicyclic group contains a finitely generated subgroup with both properties. Every nonabelian subgroup of a generalized dicyclic group is generalized dicyclic.
\end{lemma}
\begin{proof}
For the hereditary statement, intersect the abelian index-two subgroup with the given nonabelian subgroup.  Write $A$ for that abelian index-two subgroup and $y\notin A$ for an element of order four inverting $A$. Every element outside $A$ is $xy$ with $x\in A$, and $(xy)^2=y^2$ and $xy$ inverts $A$, because $A$ is abelian and $y$ inverts it; so every element outside $A$ has order four and inverts $A$. An exterior element has order four, has its nontrivial square in the intersection, and inverts that intersection, so the subgroup is again generalized dicyclic.
For the first statement fix noncommuting $a,b$. We argue by compactness, because failure of the conclusion would make every finite piece look generalized dicyclic. If every finitely generated subgroup containing them were generalized dicyclic, use the compact product $\{a,b\}\times\{0,1\}^G$. Impose that the first coordinate $t$ has order four, that the second coordinate is a homomorphism $\chi:G\to C_2$ with $\chi(t)=1$, that its kernel is abelian, and that $t$ inverts it. Every finite family is satisfied in a finitely generated subgroup containing its coefficients and $a,b$: at least one of $a,b$ lies outside that subgroup's abelian index-two subgroup. Compactness then gives the corresponding structure on all of $G$, contradicting the hypothesis.
\end{proof}

\begin{lemma}[Layered connection sets over a rigid core]\label{lem:layered}
Let $G$ be an uncountable group of cardinality $\kappa$, and let $m=76$ (Table~\ref{tab:constants}). Let $H\le G$ be a countably infinite subgroup, nonabelian unless $G$ is Boolean, and let $U\subseteq H\setminus\{1\}$ be symmetric, with $\Cay(H,U)$ a GRR and a connected robust-marker edge relation in $\Cay(H,U)$. Enumerate $G=\{g_\beta:\beta<\kappa\}$ and let
\[
 G_0=H,\qquad G_{\alpha+1}=\cl{G_\alpha,a_\alpha},\qquad G_\lambda=\bigcup_{\alpha<\lambda}G_\alpha\ \text{at limits},
\]
where $a_\alpha$ is the least-indexed $g_\beta\notin G_\alpha$; then $G_\kappa=G$. For each $\alpha<\kappa$ let $W_\alpha$ come from Lemma~\ref{lem:packet} with $M=G_\alpha$ and $a=a_\alpha$ (from Lemma~\ref{lem:booleanpacket} when $G$ is Boolean), so that
\[
 a_\alpha\in W_\alpha\subseteq\cl{H,a_\alpha}\setminus G_\alpha,\qquad |W_\alpha|\le m.
\]
Let $(\eps_\alpha)_{\alpha<\kappa}$ be successor-alternating and define $S$ by \eqref{eq:S}. Then $S=S^{-1}$ and $S\cap H=U$. For $\alpha<\kappa$, $x\in G_\alpha$ and $y\in G_{\alpha+1}\setminus G_\alpha$, the element $yx^{-1}$ lies in $G_{\alpha+1}\setminus G_\alpha$, and $x,y$ are adjacent exactly when $\eps_\alpha\oplus\mathbf1_{W_\alpha}(yx^{-1})=1$; every vertex lies in at most $m$ such pairs $\{x,y\}$ with $yx^{-1}\in W_\alpha$. Moreover:
\begin{enumerate}[label=(\roman*),nosep]
\item\label{lem:layered-aut} $\Aut(\Cay(G,S))=R(G)$;
\item\label{lem:layered-diff} every proper stage $G_\alpha$ has $\kappa$ common neighbors and $\kappa$ common nonneighbors outside $G_\alpha$; hence $S$ generates $G$, and $\Cay(G,S)$ and its complement have diameter at most two;
\item\label{lem:layered-iso} if $S'$ is defined by \eqref{eq:S} from the same $H$, $U$, $(G_\alpha)$, $(W_\alpha)$ and another such bit sequence $(\eps'_\alpha)$, then every isomorphism $\Cay(G,S)\to\Cay(G,S')$ maps the partition of $G$ into right $G_\alpha$-cosets onto itself for every $\alpha\le\kappa$; if such an isomorphism exists, then $\eps'_\alpha=\eps_\alpha$ for every $\alpha<\kappa$.
\end{enumerate}
\end{lemma}

For \ref{lem:layered-aut}, Corollary~\ref{cor:cayleyrecovery} recovers the right $G_\alpha$-coset partitions, the GRR core makes a rooted automorphism $H$-equivariant, the packet lemmas fix each retained generator, and normalizing at every base point extends this to all of $G$.

\begin{proof}
Since $|G_\alpha|\le\max\{\aleph_0,|\alpha|\}<\kappa$, the next choice always exists. Induction shows that $G_\alpha$ contains every $g_\beta$ with $\beta<\alpha$, so $G_\kappa=G$. The chain is strict and has length exactly $\kappa$, a limit ordinal, so there is no last layer. The identity $S=S^{-1}$ was noted with \eqref{eq:S}, and $S\cap H=U$ because the layers miss $G_0=H$. For $x\in G_\alpha$ and $y\in G_{\alpha+1}\setminus G_\alpha$, $yx^{-1}\in G_{\alpha+1}$, and $yx^{-1}\notin G_\alpha$ because $x\in G_\alpha$ and $y\notin G_\alpha$; so \eqref{eq:S} gives its adjacency. For a fixed vertex, distinct partners give distinct elements $yx^{-1}$, so at most $|W_\alpha|\le m$ of them lie in $W_\alpha$.

For \ref{lem:layered-aut}, let $F$ be an automorphism of $\Cay(G,S)$ fixing $1$. Since $H$ is infinite, $|H|>2m$, and Corollary~\ref{cor:cayleyrecovery} shows that $F$ preserves every right $G_\alpha$-coset partition, every layer and every packet $W_\alpha$. For each $x\in G$, the coordinate map $hx\mapsto h$ identifies the induced
graph on $Hx$ with $\Cay(H,U)$, because
$(h_1x)(h_2x)^{-1}=h_1h_2^{-1}$. Write $F(Hx)=Hy$. In these coordinates,
$h\mapsto F(hx)y^{-1}$ is an automorphism of $\Cay(H,U)$.
Since this core graph is a GRR, that map has the form $h\mapsto hk$ for some
$k\in H$. Hence $F(hx)=hky$ and $F(x)=ky$, proving
\begin{equation}\label{eq:Hequiv}
 F(hx)=hF(x)\qquad\text{for all }h\in H,x\in G.
\end{equation}

By \eqref{eq:Hequiv}, $F\in\mathcal E_H$, and the preceding paragraph gives
$F(W_\alpha)=W_\alpha$. In the nonabelian case, Lemma~\ref{lem:packet} with
$M=G_\alpha$ therefore fixes $Ha_\alpha$ pointwise; Lemma~\ref{lem:booleanpacket}
gives the same conclusion in the Boolean case. In particular,
$F(a_\alpha)=a_\alpha$ for every retained generator. The argument so far gives this only at the identity, and we need the same conclusion at every base point. For every $x\in G$, the normalized map
\[
 F_x(y)=F(yx)F(x)^{-1}
\]
is a graph automorphism fixing $1$. Everything proved for the rooted map therefore applies to $F_x$, giving
$F_x(a_\alpha)=a_\alpha$, or $F(a_\alpha x)=a_\alpha F(x)$ at every base point. Exact $H$-steps and all retained generator steps generate $G$, so these local identities propagate to $F(y)=yF(1)$. The automorphism group is therefore $R(G)$.

For \ref{lem:layered-diff}, fix $\alpha<\kappa$ and choose finite $n$ with $2^n>m$. For $\beta\ge\alpha+n$, strict successor inclusions give $[G_\beta:G_\alpha]\ge2^n>m$. The layer coset $a_\beta G_\beta$ contains more than $m$ left $G_\alpha$-cosets, whereas $W_\beta G_\alpha$ is the union of at most $m$ such cosets. Choose
\[
 t\in a_\beta G_\beta\setminus W_\beta G_\alpha.
\]
For every $f\in G_\alpha$, the difference $tf^{-1}$ is in layer $\beta$ and outside $W_\beta$, so its adjacency is the baseline bit $\eps_\beta$. Successor alternation leaves $\kappa$ later layers of each bit, and one witness from each, distinct layers giving distinct witnesses, gives $\kappa$ common neighbors and $\kappa$ common nonneighbors of $G_\alpha$ outside $G_\alpha$. Every pair of vertices lies in some $G_\alpha$, so every distinct pair has a common neighbor and a common nonneighbor; hence $S$ generates $G$, and $\Cay(G,S)$ and its complement have diameter at most two.

For \ref{lem:layered-iso}, Corollary~\ref{cor:cayleyrecovery} recovers the right $G_\alpha$-coset partitions and the baseline bits of $\Cay(G,S)$ and of $\Cay(G,S')$ through the ordinal-indexed recovery of Theorem~\ref{thm:recovery}. By the canonicity in that theorem, every isomorphism $\Cay(G,S)\to\Cay(G,S')$ maps each right $G_\alpha$-coset partition onto itself and preserves the baseline bit at each $\alpha<\kappa$, so $\eps'_\alpha=\eps_\alpha$ when such an isomorphism exists.\end{proof}

\begin{lemma}[Countable cores with finite prescriptions]\label{lem:core}
Let $G$ be an infinite group in a regular case, and let $A,R\subseteq G\setminus\{1\}$ be finite, disjoint and symmetric. There are a countably infinite subgroup $H\le G$ containing $A\cup R$, in a regular case and nonabelian unless $G$ is Boolean, and a symmetric set $U\subseteq H\setminus\{1\}$ with $A\subseteq U$ and $U\cap R=\varnothing$, such that $\Cay(H,U)$ is a GRR whose robust-marker edge relation $M_{76}$ is connected.
\end{lemma}
\begin{proof}
If $G$ is nonabelian, let $H_0$ be the finitely generated witness of Lemma~\ref{lem:witness}; if $G$ is Boolean, let $H_0=\{1\}$. Let $H$ be the subgroup generated by $H_0$, $A\cup R$ and a countably infinite subset of $G$. It is countably infinite. In the nonabelian case it contains $H_0$, so it is nonabelian, and the hereditary statement prevents it from being generalized dicyclic; in the Boolean case it is Boolean. Thus $H$ is in a regular case, and Theorem~\ref{thm:inversepairindex} gives $\Xi(H)=R(H)$. Theorem~\ref{thm:countable}, applied to $H$ with the prescriptions $A$ and $R$ and with $m=76$, supplies a symmetric $U\subseteq H\setminus\{1\}$ with $A\subseteq U$, $U\cap R=\varnothing$, $\Aut(\Cay(H,U))=\Xi(H)$, so $\Cay(H,U)$ is a GRR. That theorem also puts every finite vertex set of $\Cay(H,U)$ in a finite induced $76$-robust subgraph, so by Lemma~\ref{lem:robustunion}(iii), as $H$ has at least two vertices, its robust-marker edge relation $M_{76}$ is connected. \end{proof}

\subsection{Varying the limit bits}
Let $G$ be uncountable of cardinality $\kappa$, in a regular case, and let $A,R\subseteq G\setminus\{1\}$ be finite, disjoint and symmetric. Fix once and for all the core $H,U$ given by Lemma~\ref{lem:core} for $A,R$, and over it the chain $(G_\alpha)_{\alpha\le\kappa}$, with retained generators $a_\alpha$, and the packets $W_\alpha$ of Lemma~\ref{lem:layered}; only the free bits will vary.

Let $\Lambda$ consist of zero and all nonzero limit ordinals below $\kappa$. It has cardinality $\kappa$. The free bits at these indices are the coordinates that will encode the isomorphism type. Every $\alpha<\kappa$ has a unique expression $\alpha=\lambda+n$ with $\lambda\in\Lambda$ and $n<\omega$. For a map $b:\Lambda\to\{0,1\}$ set
\[
 \eps^b_{\lambda+n}=b(\lambda)+n\pmod2,
 \qquad
 \mathbf1_{S_b}(g)=\eps^b_{r(g)}\oplus\mathbf1_{W_{r(g)}}(g)
 \quad\text{for all }g\notin H,
\]
with $S_b\cap H=U$. All core and packet choices are independent of $b$.

\begin{theorem}[The prescribed uncountable cube]\label{thm:cube}
Let $G$ be an uncountable regular-case group of cardinality $\kappa$, let $A,R\subseteq G\setminus\{1\}$ be finite, disjoint and symmetric, and let $S_b$, for $b\in\{0,1\}^{\Lambda}$, be the connection sets defined above over the core of Lemma~\ref{lem:core} for $A,R$ and the chain and packets of Lemma~\ref{lem:layered}. Then:
\begin{enumerate}[label=(\roman*),nosep]
\item\label{thm:cube-cube} with the product topologies of Section~\ref{sec:notation}, the map $b\mapsto S_b$ is a topological embedding of $\{0,1\}^{\Lambda}$ into $\mathcal X_G$ whose image is closed and lies in the cylinder $\{S\in\mathcal X_G:A\subseteq S,\ S\cap R=\varnothing\}$; since $|\Lambda|=\kappa$, the image is a closed copy of $\{0,1\}^{\kappa}$;
\item\label{thm:cube-grr} every $\Cay(G,S_b)$ is a GRR, and distinct $b$ give nonisomorphic uncolored graphs;
\item\label{thm:cube-witness} in every $\Cay(G,S_b)$, each proper stage $G_\alpha$, and hence each $F\subseteq G$ with $|F|<\cf(\kappa)$, has $\kappa$ common neighbors and $\kappa$ common nonneighbors outside itself; in particular every distinct pair has $\kappa$ common neighbors and $\kappa$ common nonneighbors, and $\Cay(G,S_b)$ and its complement both have diameter two.
\end{enumerate}
\end{theorem}

\begin{proof}
Each sequence $(\eps^b_\alpha)$ is successor-alternating, so Lemma~\ref{lem:layered} applies to every $S_b$, with the same core, chain and packets.

\emph{Translation recovery.} By Lemma~\ref{lem:layered}\ref{lem:layered-aut}, $\Aut(\Cay(G,S_b))=R(G)$, and \eqref{eq:sandwich} gives equality with $\Xi(G)$; so every graph of the family is a GRR.

\emph{Isomorphism recovery.} If $\Cay(G,S_b)\cong\Cay(G,S_{b'})$, Lemma~\ref{lem:layered}\ref{lem:layered-iso} gives $\eps^b_\alpha=\eps^{b'}_\alpha$ for every $\alpha<\kappa$; at $\alpha=\lambda\in\Lambda$ this reads $b(\lambda)=b'(\lambda)$, so $b=b'$. This proves \ref{thm:cube-grr}. Every $\alpha<\kappa$ is uniquely $\lambda+n$ with $\lambda\in\Lambda$ and $n<\omega$, so $\kappa=|\Lambda|\cdot\aleph_0$; since $\kappa$ is uncountable, $|\Lambda|=\kappa$.

\emph{Continuity and bit recovery.} For $g\notin H$ of rank $\lambda+n$, with $\lambda\in\Lambda$ and $n<\omega$,
\[
 \mathbf1_{S_b}(g)=b(\lambda)\oplus(n\bmod2)\oplus\mathbf1_{W_{\lambda+n}}(g).
\]
Each such coordinate depends on one input bit, while $S_b\cap H=U$ makes the coordinates inside $H$ constant, so $b\mapsto S_b$ is continuous. The packet generator $a_\lambda$ lies in $W_\lambda$, so $b(\lambda)=1-\mathbf1_{S_b}(a_\lambda)$; the map is injective, and its inverse on the image is continuous. The domain is compact and $\mathcal X_G$ is Hausdorff, so the image is closed. Since $A\cup R\subseteq H$, $A\subseteq U$ and $U\cap R=\varnothing$, the identity $S_b\cap H=U$ places every $S_b$ in the cylinder. This proves \ref{thm:cube-cube}.

\emph{Later-layer witnesses.} Lemma~\ref{lem:layered}\ref{lem:layered-diff} gives the assertion for each proper stage. If $|F|<\cf(\kappa)$, the ranks of the elements of $F\setminus H$ are bounded below $\kappa$, so $F\subseteq G_\alpha$ for some $\alpha<\kappa$. Pairs have size $2<\cf(\kappa)$, so each graph and its complement have diameter at most two; neither is complete, as $\Aut(\Cay(G,S_b))=R(G)\ne\Sym(G)$, so both diameters are two.
\end{proof}

\begin{theorem}[Regular-case existence]\label{thm:regularExist}
Every infinite regular-case group has a symmetric generating set $S$ for which
\[
\Aut(\Cay(G,S))=R(G)=\Xi(G),\qquad\diam\Cay(G,S)\le2.
\]
\end{theorem}

\begin{proof}[Proof of Theorem~\ref{thm:regularExist}]
For countable $G$, Theorem~\ref{thm:countable} supplies a connection set $S$ with $\Aut(\Cay(G,S))=\Xi(G)$ and $\diam\Cay(G,S)=2$, and $\Xi(G)=R(G)$ by Theorem~\ref{thm:inversepairindex}. For uncountable $G$, Theorem~\ref{thm:cube}, with $A=R=\varnothing$, supplies $S=S_b$ with $\Aut(\Cay(G,S))=R(G)$ and $\diam\Cay(G,S)=2$, and \eqref{eq:sandwich} gives equality with $\Xi(G)$. In both cases diameter two makes $S$ generate $G$.
\end{proof}

\begin{proof}[Proof of Theorem~\ref{thm:watkins}]
Let $G$ be an infinite group that is neither abelian of exponent greater than two nor generalized dicyclic. If $G$ is abelian, it has no element of order greater than two, so $G$ is Boolean. Otherwise $G$ is nonabelian and not generalized dicyclic. In either case $G$ is in a regular case, and Theorem~\ref{thm:regularExist} gives a symmetric generating set $S$ with $\Aut(\Cay(G,S))=R(G)$, so $\Cay(G,S)$ is a graphical regular representation of $G$.
\end{proof}

\begin{remark}
Both exclusions in Theorem~\ref{thm:watkins} are necessary: if $G$ is abelian of exponent greater than two, inversion, and if $G$ is generalized dicyclic, with $A$ the abelian subgroup of index two of the definition, the map fixing $A$ pointwise and inverting every element outside $A$, is a nontrivial automorphism of $G$ sending each element to itself or its inverse. It therefore preserves every connection set $S$ and is a nontrivial automorphism of $\Cay(G,S)$ fixing $1$, so neither class has a graphical regular representation, whatever its cardinality.
\end{remark}

\begin{corollary}[Exact domination and large homogeneous sets]\label{cor:domination}
Let $G,\kappa$ and $S_b$ be as in Theorem~\ref{thm:cube},
and let $\Gamma=\Cay(G,S_b)$. Then
\[
\gamma(\Gamma)=\gamma(\overline\Gamma)=\cf(\kappa),\qquad
\omega(\Gamma)=\operatorname{ind}(\Gamma)=\chi(\Gamma)=\kappa.
\]
Here $\gamma$ is the least size of a dominating set (every vertex outside it has a neighbor in it), and $\omega,\operatorname{ind},\chi$ denote clique, independence, and chromatic numbers, respectively. The displayed clique and independent-set cardinalities are attained.
\end{corollary}
\begin{proof}
Theorem~\ref{thm:cube}\ref{thm:cube-witness} gives a common nonneighbor outside every set of size less than $\delta=\cf(\kappa)$, so no such set dominates; its common-neighbor statement gives the same lower bound for the complement. To match that lower bound, choose, along the chain of length $\kappa$ of Lemma~\ref{lem:layered}, a cofinal family of bit-one layers of size $\delta$ and take $m+1$ vertices from each. Above any given vertex, at most $m$ of these $m+1$ fail to be adjacent to it, by the layer rule of Lemma~\ref{lem:layered}. Their union is therefore a dominating set of size $\delta$. Bit-zero layers give the complementary upper bound.

For the large homogeneous sets, choose one vertex from each bit-one layer. In rank order, each vertex has at most $m$ earlier nonneighbors by the layer rule of Lemma~\ref{lem:layered}: each such difference belongs to its own packet, and distinct earlier vertices give distinct differences. Greedily color this nonadjacency graph with $m+1$ colors. One color class has size $\kappa$; that class is a clique of $\Gamma$. Repeating the argument with bit-zero layers and adjacency gives an independent set of size $\kappa$, and the clique forces chromatic number $\kappa$.
\end{proof}
The domination formula also identifies the sharp common-nonneighbor threshold for these families: it is $<\cf(\kappa)$. At regular $\kappa$ this covers every smaller set; at singular $\kappa$, a dominating set of size $\cf(\kappa)$ prevents extending the assertion to all sets of size $<\kappa$.

\begin{corollary}[Connectivity after vertex deletion]\label{cor:connectivity}
Let $G,\kappa$ and $S_b$ be as in Theorem~\ref{thm:cube},
and let $\Gamma=\Cay(G,S_b)$. Deleting fewer than $\kappa$ vertices leaves both the graph and its complement of diameter at most two. Each distinct pair has $\kappa$ internally vertex-disjoint two-edge paths before deletion.
\end{corollary}
\begin{proof}
Every common neighbor supplies one of the required two-edge paths, with distinct common neighbors giving internally vertex-disjoint paths. By Theorem~\ref{thm:cube}\ref{thm:cube-witness}, applied to the pair, each distinct pair has $\kappa$ common neighbors. Removing fewer than $\kappa$ vertices cannot exhaust a set of $\kappa$ such witnesses, so the surviving graph still has diameter at most two. The same argument applies to the complement.
\end{proof}

\section{Exceptional families: abundance and diameter two}\label{sec:exceptions}
The exceptional families require a different decoder because their unavoidable inverse-pair symmetries are larger than the right-regular action. We use the finite-color result proved in Appendix~\ref{app:finitecolor}. Countably infinite exceptional groups already satisfy the abundance conclusion by Theorem~\ref{thm:countable}; for uncountable groups we first build decoders, then refine and complement them, so the intermediate diameter bounds below are not the final existential bound.

For a group $K$ and a coloring $c$ of $K\setminus\{1\}$, $\Aut_c(K)$ is the group of permutations $F$ of $K$ with $c(F(x)F(y)^{-1})=c(xy^{-1})$ for all distinct $x,y\in K$, so the colors constrain every pair of distinct vertices; Appendix~\ref{app:finitecolor} gives the formal version.

Let $G$ be uncountable of cardinality $\kappa$, with $A=G$ in the abelian case and $G=\cl{A,q}$, $A$ abelian of index two and $z=q^2$, in the generalized dicyclic case, and put $\operatorname{Sq}(A)=\{x^2:x\in A\}$. The cases are treated as follows.
\begin{center}
\begin{tabular}{>{\raggedright\arraybackslash}p{0.22\linewidth}>{\raggedright\arraybackslash}p{0.25\linewidth}>{\raggedright\arraybackslash}p{0.25\linewidth}>{\raggedright\arraybackslash}p{0.18\linewidth}}
\hline
algebraic situation & decoder & full complement codegrees & binary refinement\\
\hline
$G$ abelian, $|\operatorname{Sq}(A)|=\kappa$ & codegree bands, Lemma~\ref{lem:abelianbands} & Lemma~\ref{lem:complement} & $(c_0,b)$ on $G$\\
$G$ abelian, $|\operatorname{Sq}(A)|<\kappa$ & Lemma~\ref{lem:abeliansplit}(i), then Lemma~\ref{lem:booleandecoder} & untagged $B$-coordinates, proof of Proposition~\ref{prop:exceptmany} & $(c_0,b)$ on $K_0$\\
$G$ dicyclic, $|\operatorname{Sq}(A)|=\kappa$ & codegree bands, Lemma~\ref{lem:abelianbands} & Lemma~\ref{lem:complement} & $(c_0,b)$ on $G$, $b(\invp z)=0$\\
$G$ dicyclic, $|\operatorname{Sq}(A)|<\kappa$ & Lemma~\ref{lem:abeliansplit}(ii), then Lemma~\ref{lem:booleandecoder} & untagged $B$-coordinates, proof of Proposition~\ref{prop:exceptmany} & $(c_0,b)$ on $K_0$\\
\hline
\end{tabular}
\end{center}
In the last column $c_0$ is a coloring from Theorem~\ref{thm:finitecolor} and $b$ a binary function on the nonidentity inverse pairs; Proposition~\ref{prop:exceptmany} adds this binary coordinate in each case.

\subsection{A finite Boolean-factor decoder}
Let $B=C_2^{18}$, with basis $e_0,\ldots,e_{17}$. The finite decoder starts from the graph $\mathcal K$ on $\{0,\ldots,17\}$ obtained by subdividing $K_4$. Its branch vertices are $a=0$, $b=1$, $c=2$, $d=3$, its direct edges are $01$, $02$, $13$, and its remaining edges form the paths
\[
 0-4-3,\qquad 1-5-6-2,\qquad
 2-7-8-9-10-11-12-13-14-15-16-17-3.
\]
Thus the internally disjoint branch-to-branch paths $ab$, $ac$, $ad$, $bc$, $bd$, $cd$ have lengths $1$, $1$, $2$, $3$, $1$, $12$.
\begin{lemma}[The seed graph]\label{lem:seed}
The graph $\mathcal K$ has eighteen vertices, is triangle-free, has minimum degree two, and is rigid, that is, $\Aut(\mathcal K)=1$.
\end{lemma}
\begin{proof}
There are $4+(0+0+1+2+0+11)=18$ vertices. The branch vertices have degree three and every other vertex has degree two. The only unsubdivided edges are $ab$, $ac$, $bd$, which contain no triangle, so $\mathcal K$ is triangle-free. An automorphism preserves the branch vertices and the lengths of the maximal paths between them. The incident path-length multisets of $a$, $b$, $c$, $d$ are respectively
\[
 \{1,1,2\},\quad \{1,1,3\},\quad \{1,3,12\},\quad \{1,2,12\}.
\]
These are distinct, so all branch vertices are fixed. Every path is then fixed pointwise. \end{proof}
Put
\[
 U_B=\{e_i:0\le i<18\}\cup\{e_i+e_j:ij\in E(\mathcal K)\}.
\]
\begin{lemma}[A Boolean GRR]\label{lem:booleanGRR}
The set $U_B=\{e_v:v\in V(\mathcal K)\}\cup\{e_u+e_v:uv\in E(\mathcal K)\}$ has $18+20=38$ elements. The Cayley graph $\Cay(B,U_B)$ is a GRR, every edge of it lies in a triangle, and its diameter is at most $18$.
\end{lemma}
\begin{proof}
The graph $\mathcal K$ has $3+2+3+12=20$ edges, so $|U_B|=18+20=38$. Inside the neighborhood of zero in $\Cay(B,U_B)$, a basis vertex $e_v$ has degree $2\deg_{\mathcal K}(v)\ge4$, whereas an edge-sum vertex has degree two. The latter assertion uses triangle-freeness: another adjacent edge-sum would complete a triangle in $\mathcal K$. The basis is therefore intrinsically recognizable, its induced graph is $\mathcal K$, and by Lemma~\ref{lem:seed} a rooted automorphism fixes the basis pointwise. For an automorphism $F$, normalization at every vertex gives $F(x+e_v)=F(x)+e_v$, forcing $F$ to be a translation. Every edge lies in a triangle: an edge $\{x,x+e_v\}$ with $x+e_u+e_v$ for a neighbor $u$ of $v$, and an edge $\{x,x+e_u+e_v\}$ with $x+e_u$. The basis gives diameter at most $18$.
\end{proof}

\begin{lemma}[Tag codes]\label{lem:tags}
Identify $B$ with $\mathbb F_2^9\times\mathbb F_2^9$ through the coordinates in the basis $e_0,\ldots,e_{17}$, and put
\[
 C=\{(x,x):x\in\mathbb F_2^9,\ \textstyle\sum_i x_i=0\}.
\]
Fix a basis vector $e\in B$ and order the coset $e+C$ by choosing the first eight coordinates of $x$ lexicographically, the ninth being their sum. Let $\Pi$ be the first $191$ vectors of $e+C$ and $\Pi_j$ its first $2j-1$ vectors, $1\le j\le96$. Then $\Pi$ consists of odd-weight vectors with pairwise Hamming distance at least four, the sets $\Pi_1\subseteq\cdots\subseteq\Pi_{96}$ are nested of odd sizes $2j-1$, and every $\Pi_j$ has trivial translation stabilizer in $B$. Moreover
\[
 \Pi+\Pi\subseteq C,\qquad C\cap U_B=\varnothing,
\]
so sums of two distinct elements of $\Pi$ have weight at least four and $(\Pi+\Pi)\cap U_B=\varnothing$.
\end{lemma}
\begin{proof}
The subspace $C$ has dimension eight. A nonzero member $(x,x)$ has weight $2\operatorname{wt}(x)\ge4$, since $\operatorname{wt}(x)$ is even and nonzero. The coset $e+C$ therefore consists of $256\ge191$ odd-weight vectors at pairwise distance at least four, and $\Pi+\Pi\subseteq C$. The elements of $U_B$ are, by its definition, the basis vectors and the sums of the two basis vectors of an edge of $\mathcal K$, so $U_B$ contains only nonzero vectors of weight one or two, while $C$ contains only zero and vectors of weight at least four; hence $C\cap U_B=\varnothing$. Each $\Pi_j$ has trivial translation stabilizer because its cardinality is odd: a nonzero translation of $B$ has no fixed point and is an involution, so it pairs off the points of any set it preserves.
\end{proof}

For a group $K_0$, choose the inverse-invariant coloring $c$ supplied by
Theorem~\ref{thm:finitecolor}. We encode each color by the size of a Boolean tag set. The coloring $c$ has at most $17$ colors. The decoder has capacity $96$, with nested tags
$\Pi_1\subseteq\cdots\subseteq\Pi_{96}$ of odd sizes $2j-1\le191$, and the binary refinement in the proof of
Proposition~\ref{prop:exceptmany} uses at most $34\le96$ colors. Regard the palette of $c$ as a subset of
$\{1,\ldots,96\}$, so that every $\Pi_{c(h)}$ is defined. On the
\emph{existing} product $K_0\times B$, set
\[
 S=(\{1\}\times U_B)\cup\{(h,t):h\ne1,\ t\in \Pi_{c(h)}\}.
\]
This set is inverse-closed: $(h,t)^{-1}=(h^{-1},t)$ because $B$ is Boolean,
and $c(h^{-1})=c(h)$.
\begin{lemma}[The Boolean-factor decoder]\label{lem:booleandecoder}
Let $K_0$ be a group and $B=C_2^{18}$. Let $U_B\subseteq B$ be the connection set of Lemma~\ref{lem:booleanGRR}, built from the seed graph $\mathcal K$ of Lemma~\ref{lem:seed}, and let $\Pi_1\subseteq\cdots\subseteq\Pi_{96}$ be the tag codes of Lemma~\ref{lem:tags}, with $|\Pi_j|=2j-1$. Let $c:K_0\setminus\{1\}\to\{1,\ldots,96\}$ be an inverse-invariant coloring with $\Aut_c(K_0)=\Xi(K_0)$, and put
\[
 S=(\{1\}\times U_B)\cup\{(h,t):h\ne1,\ t\in\Pi_{c(h)}\}\subseteq K_0\times B.
\]
Then
\[
 \Aut(\Cay(K_0\times B,S))=\Xi(K_0)\times R(B)=\Xi(K_0\times B),
\]
and $\Cay(K_0\times B,S)$ has diameter at most $19$. Moreover:
\begin{enumerate}[label=(\roman*),nosep]
\item the triangle-edge components of $\Cay(K_0\times B,S)$ are exactly the $B$-fibers $\{k\}\times B$, $k\in K_0$;
\item for $h\ne k$, every vertex of $\{h\}\times B$ has exactly $2c(hk^{-1})-1$ neighbors in $\{k\}\times B$, so these cross-fiber degrees recover $c$; hence, if $S'$ is built in the same way from a second such coloring $c'$ of $K_0$, every isomorphism $\Cay(K_0\times B,S)\to\Cay(K_0\times B,S')$ maps $B$-fibers onto $B$-fibers, and the bijection $\phi$ of $K_0$ that it induces satisfies $c'(\phi(h)\phi(k)^{-1})=c(hk^{-1})$ for all $h\ne k$;
\item for each $k\ne1$, $S\cap(\{k\}\times B)=\{k\}\times\Pi_{c(k)}$, a tag of size $2c(k)-1\le191$, so $S$ meets each fiber $\{k\}\times B$, $k\ne1$, inside a tag of size at most $191$.
\end{enumerate}
\end{lemma}

The proof reads the $B$-fibers off the graph as its triangle-edge components and the coloring $c$ off the
cross-fiber degrees (Figure~\ref{fig:decoder}). An automorphism therefore acts on the fibers by an element of $\Aut_c(K_0)$, and inside
the fibers by one common translation of $B$.
\begin{proof}
The first task is to recover the fibers intrinsically. No triangle crosses two fibers: a two-fiber triangle would give a difference in $(\Pi+\Pi)\cap U_B$, which is empty by Lemma~\ref{lem:tags}, while a three-fiber triangle would have three odd Boolean differences summing to zero. Each fiber $\{k\}\times B$ spans a copy of $\Cay(B,U_B)$, in which every edge lies in a triangle and which is connected, its diameter being at most $18$ (Lemma~\ref{lem:booleanGRR}). The triangle-edge components are therefore exactly the $B$-fibers, which is (i). Their GRRs (Lemma~\ref{lem:booleanGRR}) force an automorphism to have the form $(h,b)\mapsto(\phi(h),b+a_h)$. Cross-fiber degrees $2c(hk^{-1})-1$ then recover $c$, which is (ii); since (i) and these degrees are read off the graph alone, an isomorphism between two such graphs over $K_0$ and $B$ preserves the fibers and the colors in the same way. Thus $\phi$ preserves $c$ on differences, so $\phi\in\Aut_c(K_0)=\Xi(K_0)$ by the hypothesis on $c$. Finally, cross-neighborhood preservation gives $\Pi_j+a_h+a_k=\Pi_j$, forcing all $a_h$ equal. We obtain
\[
 \Aut(\Cay(K_0\times B,S))=\Xi(K_0)\times R(B)=\Xi(K_0\times B).
\]
For the last equality, also recorded in \cite[Lemma~4]{LdS}, a rooted inverse-pair permutation fixes every $(1,b)$ step and then restricts on the other coordinate to $\Xi(K_0)$; conversely the displayed product preserves all inverse pairs. The displayed equality identifies the full automorphism group. The diameter is at most 19: take one selected exterior step and at most 18 basis steps. Item (iii) is the definition of $S$, with $|\Pi_j|=2j-1\le191$ by Lemma~\ref{lem:tags}.
\end{proof}
\begin{figure}[!htbp]
\centering
\begin{tikzpicture}[x=1cm,y=1cm,>=Latex,every node/.style={font=\small}]
  \path[use as bounding box] (-.3,-1.05) rectangle (13.9,4.0);
  \node at (2.0,3.72) {$\{h\}\times B$};
  \node at (11.3,3.72) {$\{k\}\times B$};
  \draw[rounded corners=3pt,fill=black!4] (0,0) rectangle (4,3.30);
  \draw[rounded corners=3pt,fill=black!4] (9.05,0) rectangle (13.55,3.30);
  \node[align=center] at (2,2.68) {internal graph\\$\Cay(B,U_B)$};
  \node[align=center] at (11.3,2.84) {the same internal GRR};
  \fill (3.20,1.34) circle (2pt);
  \node[anchor=east] at (2.95,1.34) {$(h,b)$};
  \foreach \yy/\tag in {2.14/1,1.54/2}{
    \fill (9.65,\yy) circle (2pt);
    \node[anchor=west] at (9.85,\yy) {$(k,b+t_\tag)$};
    \draw[-] (3.20,1.34)--(9.65,\yy);
  }
  \node at (10.05,1.01) {$\vdots$};
  \fill (9.65,.48) circle (2pt);
  \node[anchor=west] at (9.85,.48) {$(k,b+t_{2j-1})$};
  \draw[-] (3.20,1.34)--(9.65,.48);
  \node[align=center,fill=white,inner sep=3pt] at (6.52,2.91)
    {$j=c(hk^{-1})$\\$\Pi_j=\{t_1,\ldots,t_{2j-1}\}$};
  \node[font=\footnotesize] at (2,2.03) {$|B|=2^{18}$};
  \node[align=center,font=\footnotesize] at (2,.42)
    {Every internal edge\\lies in a triangle.};
  \node[align=center,font=\footnotesize] at (6.8,-.63)
    {$\mathcal N((h,b))\cap(\{k\}\times B)=\{(k,b+t):t\in \Pi_j\}$;\quad
     cross-fiber degree $=2j-1$.};
\end{tikzpicture}
\caption{The Boolean-factor decoder of Section~\ref{sec:exceptions}, \emph{before complementation}. For distinct $h,k\in K_0$, a vertex $(h,b)$ has exactly the displayed set of neighbors in the $k$-fiber; the drawing abbreviates the list for $j\ge2$, while $j=1$ gives one neighbor. The fibers and their internal GRRs are schematic. No triangle crosses fibers, so triangle-edge components recover them; the cross-fiber degree recovers $j$, and the trivial translation stabilizer of the odd tag set $\Pi_j$ forces the fiber translations to agree.}
\label{fig:decoder}
\end{figure}
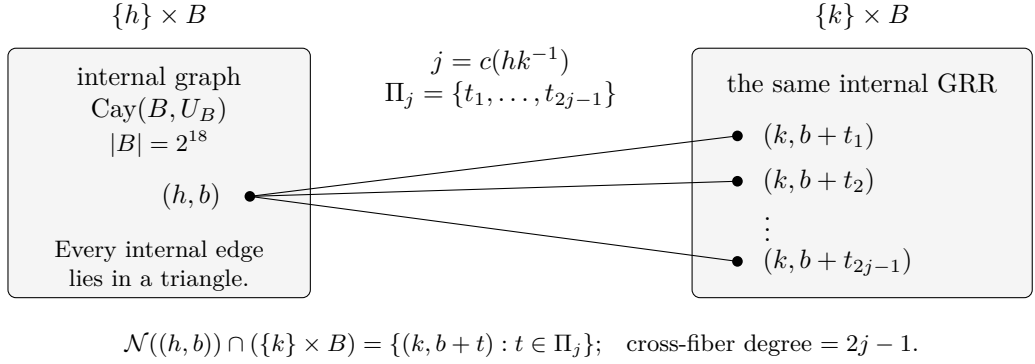

\subsection{Uncountable abelian and generalized dicyclic groups}
Let $G$ be a group of cardinality $\kappa>\aleph_0$ that is abelian or generalized dicyclic, written multiplicatively. In the abelian case put $A=G$ and $z=1$. In the dicyclic case write $G=\cl{A,q}$, with $A$ abelian of index two, $q^2=z\ne1$, $z^2=1$ and $qaq^{-1}=a^{-1}$ for $a\in A$; then $|A|=\kappa$. For $C\le A$ put $\operatorname{Sq}(C)=\{x^2:x\in C\}$, which is $2C$ in additive notation.
If $|\operatorname{Sq}(A)|=\kappa$, fix an inverse-invariant coloring $c$ of $G\setminus\{1\}$ with finitely many colors, each at least $1$, and $\Aut_c(G)=\Xi(G)$, for instance one from Theorem~\ref{thm:finitecolor}.
The idea is to encode that color in a narrow codegree band at the stage when each inverse pair first enters. Enumerate the inverse pairs outside $\cl z$ as $(\invp{g_\xi})_{\xi<\kappa}$ in order
type the initial ordinal $\kappa$, with each pair occurring once and one
representative chosen for each. Start the running subgroup at $\cl z$;
adjoin each target and then its finitely many witnesses as described below,
and take unions at limit stages. This gives a continuous chain of subgroups. For a target pair $\invp g$ of color $j$, first adjoin its representative $g$, choose $t\in A$ with $t^2$ outside the current subgroup $H$, and insert
\[
 Y(t,g)=\{t,t^{-1},tg,(tg)^{-1}\}.
\]
In the abelian case, written additively, this is $\{t,-t,t+g,-t-g\}$; for $g\notin A$ it is $\{t,t^{-1},tg,ztg\}$. All four points are new. Its only nonidentity differences lying in the current subgroup are $g$ and $g^{-1}$, represented twice each if $g\ne g^{-1}$ and four times if $g=g^{-1}$, together with $z$, represented twice, when $g\notin A$. To make the eventual count depend only on the color, use $6j$ blocks for this inverse pair if $g\ne g^{-1}$ and $3j$ blocks if $g=g^{-1}$, adjoining $t$ after each block. The multiplicity is applied once per pair, not once for each orientation. Before requirement $\xi<\kappa$ and at each of its finite insertion substages,
the subgroup has cardinality at most $\max\{|\xi|,\aleph_0\}<\kappa$:
each earlier requirement added only finitely many generators. As $|\operatorname{Sq}(A)|=\kappa$, a witness $t\in A$ with $t^2\notin H$ therefore exists even when $\kappa$ is singular; the successive witness
adjoinings are the entry-rank stages used in the count below.

\begin{lemma}[Abelian and dicyclic codegree bands]\label{lem:abelianbands}
Let $G$, $A$ and $z$ be as in the construction above, with $|G|=\kappa>\aleph_0$ and $|\operatorname{Sq}(A)|=\kappa$: either $G$ is abelian, $A=G$ and $z=1$, or $G=\cl{A,q}$ is generalized dicyclic and $z=q^2$. Let $c$ be an inverse-invariant coloring of $G\setminus\{1\}$ with finitely many colors, each at least $1$, such that $\Aut_c(G)=\Xi(G)$; a coloring from Theorem~\ref{thm:finitecolor} is one. Let $T$ be the union of the blocks $Y(t,g)$ inserted for $c$ by the construction above, and write the construction as a continuous chain $(H_\sigma)_{\sigma\le\lambda}$ of subgroups with $H_0=\cl z$ and $H_\lambda=G$, in which each $H_{\sigma+1}$ is generated by $H_\sigma$ and one element, a target or a witness; each successor $\sigma+1$ is a \emph{rank}. In this notation:
\begin{enumerate}[label=(\alph*),nosep]
\item every block $Y(t,g)$ has a rank $\sigma+1$ with $H_{\sigma+1}=\cl{H_\sigma,t}$, $g\in H_\sigma$, $t\in A$ and $t^2\notin H_\sigma$, the last possible because $|H_\sigma|<\kappa=|\operatorname{Sq}(A)|$; distinct blocks have distinct ranks, so each rank carries at most four points of $T$;
\item a target pair $\invp g$ receives $6c(g)$ blocks if $g\ne g^{-1}$ and $3c(g)$ blocks if $g=g^{-1}$;
\item the \emph{entry rank} $\rho(d)$ of $d\notin\cl z$, the least $\sigma$ with $d\in H_\sigma$, is a rank, because the chain is continuous and $d\notin H_0$.
\end{enumerate}
Then $T=T^{-1}$, $T$ contains no involution, the blocks targeted at $\invp d$ supply exactly $12c(d)$ representations $d=x^{-1}y$ with $x,y\in T$, and
\[
 12c(d)\le\nu_T(d)\le12c(d)+8\qquad(d\notin\cl z);
\]
in the dicyclic case, moreover, $\nu_T(z)=\kappa$. Two vertices $u\ne v$ of $\Cay(G,T)$ satisfy $|N(u)\cap N(v)|=\nu_T(vu^{-1})$. Hence these codegrees single out $z$ and recover $c$ on $G\setminus\cl z$, the automorphism group of $\Cay(G,T)$ is $\Xi(G)$, and it has diameter at most two.
\end{lemma}

The proof sorts the pairs $(x,y)\in T\times T$ with $x^{-1}y=d$ by the larger rank of $x$ and $y$: at most eight pairs have that rank equal to $\rho(d)$, and every other pair lies in a block targeted at $\invp d$.
\begin{proof}
For a target pair $\invp g$ of color $j$, the construction intends the following representations $d=x^{-1}y$, $x,y\in T$, of each $d\in\invp g$:
\begin{center}
\begin{tabular}{lccc}
\hline
target type & representations per designated block & blocks for color $j$ & intended total\\
\hline
$g\in A$, $g\ne g^{-1}$ & $2$ & $6j$ & $12j$\\
$g\in A$, $g=g^{-1}$ & $4$ & $3j$ & $12j$\\
exterior dicyclic, $g\notin A$ & $2$ & $6j$ & $12j$\\
\hline
\end{tabular}
\end{center}
Every other representation of $d\notin\cl z$ has the entry rank $\rho(d)$ as the larger rank of $x$ and $y$, and there are at most $8$ of these; in the dicyclic case $z$ has $\kappa$ representations.

Every element of $G\setminus A$ squares to $z$ and inverts $A$ by conjugation, and $z$ is central; hence $(tg)^{-1}=ztg$ and $gt^{-1}=tg$ for $t\in A$ and $g\notin A$. Every block is inverse-closed, so $T=T^{-1}$. Fix a block $Y(t,g)$ of rank $\sigma+1$. Its points lie in $H_{\sigma+1}\setminus H_\sigma$, since $g\in H_\sigma$ and $t\notin H_\sigma$; so $1\notin T$, and each point of $T$ has the rank of its block. For the twelve ordered pairs $(x,y)$ of different entries of the list $t,t^{-1},tg,(tg)^{-1}$, the differences $x^{-1}y$ lying in $H_\sigma$ are, with multiplicity:
\begin{center}
\begin{tabular}{ll}
\hline
target & old differences $x^{-1}y\in H_\sigma$\\
\hline
$g\in A$, $g\ne g^{-1}$ & $g$ twice, $g^{-1}$ twice\\
$g\in A$, $g=g^{-1}$ & $g$ four times\\
$g\notin A$ & $g$ twice, $g^{-1}$ twice, $z$ twice\\
\hline
\end{tabular}
\end{center}
The other differences are $t^2$, $t^{-2}$, $t^2g^2$ and $t^{-2}g^{-2}$ once each and $t^2g$ and $t^{-2}g^{-1}$ twice each if $g\in A$, and $t^2$ and $t^{-2}$ once each and $t^2g$ and $zt^2g$ twice each if $g\notin A$; they lie in $t^2H_\sigma\cup t^{-2}H_\sigma$, which misses $H_\sigma$ because $t^2\notin H_\sigma$, so their entry rank is $\sigma+1$. None of the twelve differences is $1$, as $g\ne1$ and $z\ne1$ when $g\notin A$, so the four points of a block are distinct.

The later blocks avoid the earlier differences: if $x\in T$ has rank $\sigma+1$ and $y\in T$ a smaller rank, then $y\in H_\sigma$ and $x\notin H_\sigma$, so $x^{-1}y$ and $y^{-1}x$ lie in $H_{\sigma+1}\setminus H_\sigma$ and have entry rank $\sigma+1$.

Fix $d\notin\cl z$, and in a pair $(x,y)\in T\times T$ with $x^{-1}y=d$ let $\tau=\sigma+1$ be the larger of the ranks of $x$ and $y$. If the ranks differ, then $\tau=\rho(d)$ by the avoidance. If they agree, then $x,y$ lie in one block of rank $\tau$, by (a); either $d\notin H_\sigma$ and $\tau=\rho(d)$, or $d$ is an old difference of that block, hence by the table of old differences its target or the inverse of its target, since $d\ne z$, and the block is targeted at $\invp d$. At most eight pairs have $\tau=\rho(d)$: one of $x,y$ is among the at most four points of rank $\rho(d)$, and it determines the other. By the table of old differences, (b) and $c(d^{-1})=c(d)$, the blocks targeted at $\invp d$ supply exactly the $12c(d)$ pairs of the representation-count table. Hence
\[
 12c(d)\le\nu_T(d)\le12c(d)+8\qquad(d\notin\cl z).
\]
The bands $[12j,12j+8]$ are disjoint, so $c(d)=\lfloor\nu_T(d)/12\rfloor$ for $d\notin\cl z$, and these values of $\nu_T$ are finite. In the dicyclic case $G\setminus A$ has $\kappa$ elements and so contains $\kappa$ target pairs, each receiving at least six blocks with the two points $tg$ and $ztg$ outside $A$; distinct blocks are disjoint, so $|T\setminus A|=\kappa$. Each $x\in T\setminus A$ has $x^{-1}\in T$ and $x^{-1}x^{-1}=x^{-2}=z$, so the pairs $(x,x^{-1})$ give $\kappa\le\nu_T(z)\le|G|=\kappa$. For $u\ne v$, the common-neighbor formula stated with \eqref{eq:sandwich}, for the connection set $T$, gives
\[
 |N(u)\cap N(v)|=\nu_T(vu^{-1})\qquad(u\ne v).
\]
Let $\varphi$ be a graph automorphism and $u\ne v$. Then $\nu_T(\varphi(v)\varphi(u)^{-1})=\nu_T(vu^{-1})$, so $vu^{-1}=z$ exactly when $\varphi(v)\varphi(u)^{-1}=z$, and otherwise both lie outside $\cl z$ and have the same color by the bands. Thus $c(\varphi(v)\varphi(u)^{-1})=c(vu^{-1})$, and $\varphi\in\Aut_c(G)=\Xi(G)$. Conversely, \eqref{eq:sandwich} for the connection set $T$ gives $\Xi(G)\le\Aut(\Cay(G,T))$, so every element of $\Xi(G)$ preserves adjacency. Since $c\ge1$, every $d\ne1$ has $\nu_T(d)>0$, so any two distinct vertices have a common neighbor and the diameter is at most two. Finally, $T$ contains no involution: a point of $T$ in $A$ has square $t^{\pm2}$ or $t^{\pm2}g^{\pm2}$ for its block $Y(t,g)$, which lies outside $H_\sigma$, while a point outside $A$ squares to $z\ne1$ and has order four.
\end{proof}
\begin{lemma}[Abelian splitting]\label{lem:abeliansplit}
\begin{enumerate}[label=(\roman*),nosep]
\item Let $A$ be abelian, written additively, with $|A|=\kappa>\aleph_0$ and $|2A|<\kappa$. Then $A=C\oplus B_\infty$, where $C$ is a subgroup of size smaller than $\kappa$ with $2C=2A$ that can be chosen to contain any given element $z\in A$, and $B_\infty\le A[2]$ is Boolean with $\dim B_\infty=\kappa$. In particular $A$ has a $C_2^{18}$ direct factor, and Lemma~\ref{lem:booleandecoder} applies with $B$ this factor, $K_0$ a complementary factor, and $c$ the coloring of $K_0$ given by Theorem~\ref{thm:finitecolor}.
\item Let $G=\cl{A,q}$ be generalized dicyclic, with $A$ and $z=q^2$ as in the construction above, $|G|=\kappa>\aleph_0$ and $|\operatorname{Sq}(A)|<\kappa$. Then $G=\cl{C,q}\times B_\infty$, where $C\le A$ contains $z$ and has size smaller than $\kappa$, and $B_\infty\le A$ is a central Boolean subgroup of dimension $\kappa$. In particular $G=K_0\times B$ with $B\le B_\infty$, $B\cong C_2^{18}$ and $K_0=\cl{C,q}\times B_1$ for a complement $B_1$ of $B$ in $B_\infty$, and Lemma~\ref{lem:booleandecoder} applies with these $K_0$ and $B$ and the coloring of $K_0$ given by Theorem~\ref{thm:finitecolor}.
\end{enumerate}
\end{lemma}
\begin{proof}
For (i), lift a generating set of $2A$ and let $C_0$ be its generated subgroup, of size smaller than $\kappa$. Then $2C_0=2A$ and $A=C_0+A[2]$. A Boolean vector-space complement $B_\infty$ to $C_0\cap A[2]$ in $A[2]$ gives $A=C_0\oplus B_\infty$, with $\dim B_\infty=\kappa$. Given $z\in A$, write $z=a+b$ with $a\in C_0$ and $b\in B_\infty$, put $C=C_0\oplus\cl{b}$, and replace $B_\infty$ by a vector-space complement of $\cl{b}$ in it. Then $A=C\oplus B_\infty$, $z\in C$, $|C|<\kappa$, $2C=2C_0=2A$ because $2b=0$, and still $\dim B_\infty=\kappa$. Thus $A$ already has a $C_2^{18}$ direct factor $B\le B_\infty$, and Lemma~\ref{lem:booleandecoder} applies with $K_0$ a complementary factor of $B$ and $c$ the coloring of $K_0$ supplied by Theorem~\ref{thm:finitecolor}, which takes at most $17\le96$ values, so the decoder's hypothesis holds.

For (ii), $|A|=\kappa$ and $|\operatorname{Sq}(A)|<\kappa$, so (i), applied to $A$ written multiplicatively with $C$ chosen to contain $z$, gives $A=C\times B_\infty$ with $z\in C$, $|C|<\kappa$ and $B_\infty$ Boolean of dimension $\kappa$. Conjugation by $q$ inverts $A$, and inversion fixes every element of order two, so $q$ centralizes the Boolean complement $B_\infty$, which is therefore central in $G$. As $q$ normalizes $C$ and $q^2=z\in C$, $\cl{C,q}=C\cup Cq$, so $\cl{C,q}\cap A=C$. Hence $\cl{C,q}\cap B_\infty=C\cap B_\infty=\{1\}$ and $\cl{C,q}B_\infty\supseteq CB_\infty\cup CB_\infty q=G$; as $B_\infty$ is central, both factors are normal, and
\[
 G=\cl{C,q}\times B_\infty.
\]
Write $B_\infty=B\times B_1$ with $B\cong C_2^{18}$, and put $K_0=\cl{C,q}\times B_1$, so that $G=K_0\times B$. Theorem~\ref{thm:finitecolor}, applied to $K_0$, gives an inverse-invariant coloring $c$ of $K_0\setminus\{1\}$ with values in $\{1,\ldots,96\}$ and $\Aut_c(K_0)=\Xi(K_0)$. Hence $K_0$, $B$ and $c$ satisfy the hypotheses of Lemma~\ref{lem:booleandecoder}, and the case $|\operatorname{Sq}(A)|<\kappa$ reduces to that lemma.
\end{proof}

\subsection{Complementation and the final diameter bound}
\begin{lemma}[Complementation with a small exceptional set]\label{lem:complement}
Let $|G|=\kappa$ be infinite and $S=S^{-1}\subseteq G\setminus\{1\}$. Suppose
\[
 N=\{g\ne1:\nu_S(g)=\kappa\}
\quad\text{has size smaller than }\kappa.
\]
Then $\overline S=(G\setminus\{1\})\setminus S$ has $\nu_{\overline S}(g)=\kappa$ for every $g\ne1$.
\end{lemma}
\begin{proof}
If $|S|<\kappa$, two translates and finitely many points exclude only a smaller set, so the conclusion follows. Otherwise $|S|=\kappa$. Fix $g\ne1$. To prove that the complement has full codegree at this difference, suppose instead that $\mathcal Q=G\setminus(S\cup Sg)$ is small, and choose
\[
 r\notin\{1,g\}\cup N\cup N^{-1}g.
\]
Then
\[
 |Sr\cap S|=\nu_S(r)<\kappa,\qquad
 |Sr\cap Sg|=\nu_S(gr^{-1})<\kappa.
\]
Together with $\mathcal Q$, these two intersections cover $Sr$, a set of size $\kappa$, by a finite union of smaller sets, a contradiction. Hence $|\mathcal Q|=\kappa$. Removing the possible points $1,g$ then yields $|\overline S\cap\overline Sg|=\kappa$, which is the required complement codegree.
\end{proof}
The order $gr^{-1}$ is intentional in this calculation, and the argument uses neither commutativity nor a regular-cardinal assumption.

\subsection{Refining the finite palette gives the full abundance statement}
\begin{proposition}[Exceptional abundance]\label{prop:exceptmany}
Let $G$ be an uncountable non-Boolean abelian or generalized-dicyclic group of cardinality $\kappa$. Then $G$ has $2^\kappa$ pairwise nonisomorphic realizations of $\Xi(G)$, each with $\kappa$ common neighbors at every nonidentity difference and with diameter two.
\end{proposition}

The proof refines the coloring of Theorem~\ref{thm:finitecolor} by $2^\kappa$ binary functions $b$, encodes each
refinement as a Cayley graph by Lemma~\ref{lem:abelianbands} or Lemma~\ref{lem:booleandecoder} so that any
isomorphism between two of these graphs recovers $b$, and passes to complements to obtain $\kappa$ common neighbors.
\begin{proof}
Call the cases with $|2G|=\kappa$ for abelian $G$, and with $|\operatorname{Sq}(A)|=\kappa$ for generalized-dicyclic $G=\cl{A,q}$, the large-square cases, and put $K_0=G$ there. Otherwise (the Boolean-factor case), Lemma~\ref{lem:abeliansplit}, in its abelian and in its dicyclic part, gives $G=K_0\times B$ with $B=C_2^{18}$, and $|K_0|=\kappa$ because $B$ is finite. Fix a coloring $c_0$ of $K_0$ from Theorem~\ref{thm:finitecolor}, with at most $17$ colors and $\Aut_{c_0}(K_0)=\Xi(K_0)$. Let $b$ range over the binary functions on the nonidentity inverse pairs of $K_0$, with $b(\invp z)=0$ for every choice in the dicyclic large-square case, where $z=q^2$, and refine to $c_b=(c_0,b)$, encoded by one fixed numerical palette of at most $34$ colors. The coloring $c_b$ is inverse-invariant, has finitely many colors, each at least $1$, and satisfies $\Aut_{c_b}(K_0)=\Xi(K_0)$, since $\Xi(K_0)$ preserves every inverse-invariant coloring and $\Aut_{c_b}(K_0)\subseteq\Aut_{c_0}(K_0)=\Xi(K_0)$. Since $K_0$ has $\kappa$ nonidentity inverse pairs, there are $2^\kappa$ choices of $b$.

In the large-square cases, Lemma~\ref{lem:abelianbands} applied to $c_b$ gives a symmetric connection set $T_b$ with $\Aut(\Cay(G,T_b))=\Xi(G)$ and $\nu_{T_b}(g)\le12c_b(g)+8$ for every $g\notin\cl z$, where $\cl z$ is trivial in the abelian case. Its graph recovers $c_b$ at every distinct pair whose difference lies outside $\cl z$, and in the dicyclic case it recognizes the difference $z$ by its unique infinite codegree. As $c_b(z)=(c_0(z),0)$ for every choice, an isomorphism $\Cay(G,T_b)\to\Cay(G,T_{b'})$ preserves the $c_0$ coordinate of every difference. It therefore belongs to $\Aut_{c_0}(G)=\Xi(G)$ and preserves every inverse pair of differences, so the recovered binary coordinates agree and $b=b'$.

In the Boolean-factor case, Lemma~\ref{lem:booleandecoder}, with the tag codes $\Pi_1\subseteq\cdots\subseteq\Pi_{96}$ of Lemma~\ref{lem:tags}, applies to $c_b$ and gives a connection set $S_b\subseteq K_0\times B=G$ with $\Aut(\Cay(G,S_b))=\Xi(G)$. By its part (ii), an isomorphism $\Cay(G,S_b)\to\Cay(G,S_{b'})$ maps $B$-fibers onto $B$-fibers, and the permutation $\phi$ of $K_0$ that it induces on the fibers satisfies $c_{b'}(\phi(h)\phi(k)^{-1})=c_b(hk^{-1})$ for all $h\ne k$. The $c_0$ coordinate gives $\phi\in\Aut_{c_0}(K_0)=\Xi(K_0)$, so $\phi$ preserves every inverse pair of differences, the recovered binary coordinates agree, and $b=b'$.

Finally let $\Gamma_b$ be $\Cay(G,\overline{T_b})$, respectively $\Cay(G,\overline{S_b})$, the bar denoting the complement in $G\setminus\{1\}$. Complementation preserves automorphism groups and nonisomorphism, so the graphs $\Gamma_b$ are $2^\kappa$ pairwise nonisomorphic realizations of $\Xi(G)$. In the large-square cases the band bound gives $\nu_{T_b}(g)<\kappa$ for $g\notin\cl z$, so the set $N$ of Lemma~\ref{lem:complement} for $T_b$ lies in $\{z\}$, and that lemma gives $\nu_{\overline{T_b}}(g)=\kappa$ for every $g\ne1$. In the Boolean-factor case, part (iii) of Lemma~\ref{lem:booleandecoder} gives $S_b\cap(\{k\}\times B)\subseteq\{k\}\times\Pi$ for every $k\ne1$, with $\Pi$ the set of Lemma~\ref{lem:tags}; put $Q=B\setminus\Pi$. Fix a nonidentity $(h,y)\in K_0\times B$. Since $|\Pi|=191<|B|/2$, the sets $Q$ and $Q+y$ meet; fix $t\in Q\cap(Q+y)$. For every $k\in K_0\setminus\{1,h^{-1}\}$, the points $(k,t)$ and $(kh,t+y)$ lie in $\overline{S_b}$: both first coordinates are nonidentity, and both $B$-coordinates lie in $Q$, outside every tag. Their difference $(k,t)^{-1}(kh,t+y)$ is $(h,y)$, and there are $\kappa$ choices of $k$, so $\nu_{\overline{S_b}}((h,y))=\kappa$. In every case two distinct vertices of $\Gamma_b$ therefore have $\kappa$ common neighbors, so they are at distance at most two; the connection sets $T_b$ and $S_b$ are nonempty, so $\Gamma_b$ is not complete and has diameter two. \end{proof}

\begin{proof}[Proof of Theorem~\ref{thm:all}]
Let $G$ be an infinite group of cardinality $\kappa$. If $G$ is countable, Theorem~\ref{thm:countable} gives
$2^{\aleph_0}=2^\kappa$ pairwise nonisomorphic realizations of $\Xi(G)$ of diameter two with $\kappa$ common neighbors
at every distinct pair; for a countable regular-case group these are GRRs, since $\Xi(G)=R(G)$ by Theorem~\ref{thm:inversepairindex}.
If $G$ is uncountable and in a regular case, Theorem~\ref{thm:cube}, with $A=R=\varnothing$, gives $2^\kappa$ pairwise nonisomorphic
GRRs, which realize $\Xi(G)=R(G)$ by Theorem~\ref{thm:inversepairindex}; Theorem~\ref{thm:cube}\ref{thm:cube-witness} gives $\kappa$ common neighbors of every distinct pair $x,y$,
hence diameter at most two.
If $G$ is uncountable and
exceptional, Proposition~\ref{prop:exceptmany} gives $2^\kappa$ pairwise nonisomorphic realizations of $\Xi(G)$ with
$\kappa$ common neighbors at every nonidentity difference and diameter two.

The two optimality statements are immediate. The complete graph has automorphism group $\Sym(G)$, strictly larger than
$\Xi(G)$ for infinite $G$: a transposition between elements in different inverse pairs fixing $1$ violates the
inverse-pair condition, so diameter one is impossible for a realizing graph. Finally, $2^\kappa$ is the maximum
possible number of graphs on a fixed set of size $\kappa$, because $|[G]^2|=\kappa$.
\end{proof}

\section{Finite prescriptions and algebraic consequences}\label{sec:consequences}\subsection{Closed families with prescribed finite data}\label{sec:prescriptions}
Use the connection-set space $\mathcal X_G$ defined in Section~\ref{sec:notation}.
\begin{theorem}[A closed Cantor cube in every finite cylinder]\label{thm:finitepres}
Let $G$ be an infinite regular-case group of cardinality $\kappa$. For disjoint finite symmetric $A,R\subseteq G\setminus\{1\}$, the cylinder
\[
\{S\in\mathcal X_G:A\subseteq S,\ S\cap R=\varnothing\}
\]
contains a closed subspace homeomorphic to $\{0,1\}^{\kappa}$ whose distinct members give pairwise nonisomorphic GRRs. These graphs have $\kappa$ common neighbors at each distinct pair. Both each graph and its complement have diameter two. At uncountable cardinality the witness and domination conclusions of Theorem~\ref{thm:cube} and Corollary~\ref{cor:domination} are retained.
\end{theorem}
\begin{proof}
For countable $G$, use the complemented outputs
$S=S_b$ of Theorem~\ref{thm:countable}. To obtain
$A\subseteq S$ and $S\cap R=\varnothing$, apply that theorem
with $R$ selected and $A$ forbidden. The graphs $\Cay(G,S)$ realize $\Xi(G)=R(G)$, so they are GRRs, distinct $b$ give nonisomorphic graphs, every distinct pair has $\aleph_0=\kappa$ common neighbors, and each graph and its complement have diameter two. Its statement also makes $b\mapsto S=S_b$ a continuous injection of $\{0,1\}^{\N}=\{0,1\}^{\kappa}$ into the cylinder. Compactness makes its image closed.

For uncountable $G$, Theorem~\ref{thm:cube} gives the closed copy of $\{0,1\}^{\kappa}$ in the cylinder, the pairwise nonisomorphic GRRs, the $\kappa$ common neighbors at each distinct pair, the two diameters and the witness threshold, and Corollary~\ref{cor:domination} applies to the same family.
\end{proof}
Thus each basic finite cylinder contains the maximum number of uncolored isomorphism types in a compact, continuously parameterized family. This does not say that the full set of realizing connection sets is closed, open, comeagre, or of probability one.

\subsection{Affine isomorphisms, rooted multiplication and invariant metrics}\label{sec:formal}
The following standard consequences are included to record what the all-cardinality
supply of GRRs gives immediately.

A Cayley graph is a \emph{CI-graph on $G$} if any isomorphic Cayley graph
on $G$ is its image under a group automorphism of $G$. A group law
on the vertex set of a graph is \emph{compatible} if its right translations
are graph automorphisms.

\begin{proposition}[Affine isomorphisms]\label{prop:ci}
If $\Gamma=\Cay(G,S)$ is a GRR, every transitive subgroup of
$\Aut(\Gamma)$ is $R(G)$. Hence any isomorphism
$f:\Cay(G,S)\to\Cay(K,T)$ has the form $f(x)=\phi(x)k_0$ for a group
isomorphism $\phi:G\to K$ and $k_0\in K$; in particular $\Gamma$ is a
CI-graph on $G$ and is not a Cayley graph of a nonisomorphic group.
\end{proposition}
\begin{proof}
A transitive subgroup of a regular group contains the unique element sending a
fixed vertex to each target, so it is the whole regular group. Conjugating
$R(K)$ through $f$ and normalizing $f(1)=1$ gives
$f(xg)=f(x)f(g)$; undoing the normalization yields the stated affine form.
\end{proof}

\begin{corollary}[A rooted graph determines multiplication]
\label{cor:rooted}On the vertex set of a GRR, fixing an identity vertex determines a unique
compatible group law; with arbitrary identity $a$, the compatible law is
$x\star_a y=xa^{-1}y$, so there are exactly $|G|$ such labeled laws.
\end{corollary}

\begin{proof}
Let $\Gamma=\Cay(G,S)$ and let $\star$ be a compatible law with identity $a$.
Its right translations $x\mapsto x\star y$ are automorphisms of $\Gamma$ and
form a group acting regularly on the vertices, so by
Proposition~\ref{prop:ci} they form $R(G)$. Hence for each $y$ there is
$h\in G$ with $x\star y=xh$ for all $x$; taking $x=a$ gives $y=ah$, so
$x\star y=xa^{-1}y$. Conversely, $x\star_a y=x(a^{-1}y)$ is a group law with
identity $a$ whose right translations are those of $R(G)$. So each $a\in G$
is the identity of exactly one compatible law, and distinct identities give
distinct laws, which makes $|G|$ laws in all.
\end{proof}

\begin{corollary}[Three-valued path metrics]\label{cor:metric}
Every infinite group $G$ admits a complete, uniformly discrete, right-invariant
metric with values $0,1,2$ and isometry group $\Xi(G)$, and no metric on $G$ with only two values has isometry group $\Xi(G)$.
\end{corollary}

\begin{proof}
Use the shortest-path metric of a realizing graph of diameter two given by Theorem~\ref{thm:all}. Two total distance values would make the metric equilateral off the diagonal and give the full symmetric group as isometry group, not $\Xi(G)$.
\end{proof}

\subsection{Oriented regular representations}\label{sec:orr}

\begin{proposition}[Oriented regular representations]\label{prop:orr}
Let $G$ be a group. For $T\subseteq G\setminus\{1\}$, let $\Cay^{\to}(G,T)$ be the digraph with an arc from $x$ to $y$ exactly when $xy^{-1}\in T$. It is an \emph{oriented regular representation} (ORR) if $T$ generates $G$, $T\cap T^{-1}=\varnothing$ and its automorphism group is $R(G)$.
\begin{enumerate}[label=(\roman*),nosep]
\item If $T$ generates $G$ and $T\cap T^{-1}=\varnothing$, then $\Aut(\Cay^{\to}(G,T))\cap\Xi(G)=R(G)$.
\item If $\Cay(G,S)$ is a realization of $\Xi(G)$ with $S$ generating $G$ and containing no involution, then every orientation $T$ of $S$ ($T\cup T^{-1}=S$, $T\cap T^{-1}=\varnothing$) gives an ORR.
\item Every uncountable abelian group $A$ with $|2A|=|A|$, and every uncountable generalized dicyclic group with $|\operatorname{Sq}(A)|=|A|$, admits an ORR.
\item A generalized dihedral group $A\rtimes C_2$, with $C_2$ acting by inversion, admits no ORR.
\end{enumerate}
\end{proposition}

\begin{proof}
(i) For $t\in T$ and $x\in G$ the pair $(tx,x)$ is an arc, so an automorphism $F\in\Xi(G)$ satisfies $F(tx)F(x)^{-1}\in\invp{t}\cap T=\{t\}$; hence $F(tx)=tF(x)$ for all $t\in T$, therefore for all $t\in\langle T\rangle=G$, and $F$ is the right translation by $F(1)$.

(ii) The orientation $T$ generates $G$, since $T\cup T^{-1}=S$ does, and $\Aut(\Cay^{\to}(G,T))\subseteq\Aut(\Cay(G,S))=\Xi(G)$; so (i) gives $\Aut(\Cay^{\to}(G,T))=R(G)$.

(iii) The connection sets $S$ of Lemma~\ref{lem:abelianbands}, the decoders of the proof of Proposition~\ref{prop:exceptmany} taken before complementation, satisfy the hypotheses of (ii): by that lemma $\Cay(G,S)$ realizes $\Xi(G)$. In the abelian case of Section~\ref{sec:exceptions}, with $|2A|=\kappa$, every inserted point $s\in\{t,t+a,-t,-t-a\}$ has $2s\in\{\pm2t,\pm(2t+2a)\}$, and $2t$ lies outside a subgroup containing $2a$, so $2s\ne0$. In the generalized dicyclic case, with $|\operatorname{Sq}(A)|=\kappa$, the squares of the inserted points are $t^{\pm2}$, $t^{\pm2}a^{\pm2}$, and $(tg)^2=(ztg)^2=z$, none equal to $1$: as $2t$ in the abelian case, $t^2$ lies outside a subgroup containing $a^2$, so $t^{\pm2}\ne1$ and $t^{\pm2}a^{\pm2}\ne1$, and $z\ne1$. In both cases every nonidentity codegree is positive by Lemma~\ref{lem:abelianbands}, since each lower band $12c(g)$ is at least $12$ and $\nu_S(z)$ is infinite, so the connection set generates.

(iv) Every element outside $A$ has order two, and each generating set of $A\rtimes C_2$ contains an element outside $A$, hence an involution $t=t^{-1}$. The set $T$ of an ORR generates $G$ and has $T\cap T^{-1}=\varnothing$, so it contains no involution, and the group admits no ORR.
\end{proof}

\begin{remark}[Oriented representations]\label{rem:orr}
Parts (i) and (ii) of Proposition~\ref{prop:orr} are the orientation step of \cite[Lemma~5]{LdSJAC}. A group that is not generalized dihedral is generated by its non-involutions \cite[Proposition~5.2]{Babai78}, and the converse is immediate. Finitely generated groups were classified in \cite[Proposition~7.2]{LdSAHL} and finite groups in \cite{MorrisSpigaORR}. We ask whether every infinite group that is not generalized dihedral admits an involution-free realization of $\Xi(G)$, and hence an ORR. The regular-case construction of Section~4 does not answer this directly: removing the involutions from a bit-one layer creates error pairs whose degree is $|Hg\cap\operatorname{Inv}(G)|$ for $g$ in that layer, which is infinite for every admissible countable core in, for example, $D_\infty\times B\times\mathbb Z$ with $B$ an uncountable Boolean group.
\end{remark}

\section{Computational checks and further questions}\label{sec:checks}
The construction recovers, in order, the robust initial classes, the adjacency baseline, the ordinal hierarchy, the finite packets, and the translation action.

The ancillary files contain checking programs, finite input data, and recorded
outputs for the packet construction, robust markers, baseline and twin-quotient
recovery, finite countable-schedule prefixes, and the Boolean decoder. The
protected-marker suite tests preservation of frozen codegrees during pool and
star insertions in infinite groups, including a countable locally finite
nonabelian example. Further checks examine finite hierarchy examples, robustness
margins, finite-prescription constraints, and the Boolean tag sets. These finite checks do not verify
the infinite schedules, transfinite recovery, or cardinal arguments.

The main further questions concern efficient recognition within restricted finite graph classes and finite-prescription abundance for the exceptional families. As shown in Proposition~\ref{prop:robust-conp}, testing $m$-robustness of an unrestricted finite graph is coNP-complete for each fixed $m\ge 1$. In the stability result it is the hierarchy, rather than the full automorphism group, that persists after the allowed edits. The closed-family theorem concerns selected compact families, not genericity of all realizing sets. The involution-free realization question of Remark~\ref{rem:orr} is a third.

\paragraph{Code and data availability.}
The checking programs, finite input data, and recorded outputs are supplied with
this article as ancillary or supplementary material. The programs are separate
from the \LaTeX{} build.

\paragraph{AI assistance in the research process.}
OpenAI's ChatGPT provided substantial assistance with developing arguments,
locating references, and preparing the checking programs described above. Anthropic's Claude, used through an automated revision pipeline under the author's direction, checked the arguments by independent line-by-line readings.
The checks evaluate finite examples, finite schedule prefixes, and exact
integer-matrix certificates; they do not constitute formal verification of
the infinite or transfinite arguments. Assistance with manuscript preparation
is disclosed before the references.

\appendix
\section{A finite coloring that realizes all inverse-pair symmetries}\label{app:finitecolor}
This appendix isolates the part of the finite-valued construction in \cite[Sections~3--6]{FC} needed by Section~\ref{sec:exceptions}. The naturality classification and the other metric constructions in that manuscript are not used. In particular, the proof below does not depend on any uncolored realization theorem proved earlier in this paper.

For an inverse-invariant coloring $c:K\setminus\{1\}\to\mathcal C$, its
\emph{complete difference-color-preserving permutation group} is
\[
\Aut_c(K)=\bigl\{F\in\Sym(K):
 c(F(x)F(y)^{-1})=c(xy^{-1})\text{ for all }x\ne y\bigr\}.
\]
Here inverse-invariance means $c(g)=c(g^{-1})$, so the color
$c(xy^{-1})$ is well defined on the unordered pair $\{x,y\}$.
The word \emph{complete} means that the condition is imposed on every pair
of distinct vertices, not only on the edges of a chosen Cayley graph.

\begin{theorem}[Finite-color realization]\label{thm:finitecolor}
For every group $K$ there is an inverse-invariant coloring
\[
c:K\setminus\{1\}\longrightarrow\{1,\ldots,17\}
\]
whose complete difference-color-preserving permutation group satisfies
$\Aut_c(K)=\Xi(K)$.
\end{theorem}

Plan of the proof: the colors mark $S$, so a color-preserving permutation preserves word length in $\Cay(K,S)$;
Lemma~\ref{lem:fc-pairs} at every base point controls the generator steps, proper colorings of the short-word
conflicts extend this control to words of length at most three, and \eqref{eq:fc-shortword} passes from these
words to $\Xi(K)$.

\subsection{Ranks and recovery of generator inverse pairs}
The trivial group is immediate. Otherwise, we first encode a well-order into generator inverse pairs. Well-order $K\setminus\{1\}$ and retain an element only when it is outside the subgroup generated by the earlier retained elements. Write the retained sequence as $(a_\alpha)_{\alpha<\tau}$ and put
\[
K_\alpha=\langle a_\beta:\beta<\alpha\rangle,\qquad
\mathcal A=\{a_\alpha:\alpha<\tau\},\qquad
S=\mathcal A\cup\mathcal A^{-1},\qquad B_\alpha=\invp{a_\alpha}.
\]
Here $\tau$ is an ordinal, the $B_\alpha$ are disjoint, $a_\alpha\notin K_\alpha$, and $K_\tau=K$. Continuity holds at limits because words are finite. The entry rank of a nonidentity element records the first stage at which it appears: for $g\ne1$, define $r(g)$ by
\[
g\in K_{r(g)+1}\setminus K_{r(g)}.
\]
Subgroup closure gives $r(g^{-1})=r(g)$, and
\begin{equation}\label{eq:fc-ranks}
r(uv^{-1})=\max\{r(u),r(v)\}\qquad\text{when }r(u)\ne r(v).
\end{equation}
Indeed, if $r(u)<r(v)=\beta$, then $u\in K_\beta$ and $v\notin K_\beta$, so $uv^{-1}\notin K_\beta$ but belongs to $K_{\beta+1}$. This rank rule is what lets short differences reveal the order of the retained generators.
Choose labels $\eta_0=0$, $\eta_{\alpha+1}=1-\eta_\alpha$, and $\eta_\lambda=0$ at nonzero limits, and put $p(g)=\eta_{r(g)}$.

\begin{lemma}[A two-labeled well-order]\label{lem:fc-order}
Give each index $\alpha<\tau$ vertex label $\eta_\alpha$ and each unordered pair $\{\alpha,\beta\}$ label $\eta_{\max\{\alpha,\beta\}}$. The only permutation preserving both kinds of labels is the identity.
\end{lemma}
\begin{proof}
For opposite-label vertices, the pair label identifies the larger vertex. We must also recover the order between vertices carrying the same label. For same-label distinct vertices, $\alpha<\beta$ holds exactly when there is an opposite-label $\gamma$ with $\alpha<\gamma<\beta$. The forward implication uses $\gamma=\alpha+1$: equal endpoint labels imply $\alpha+1<\beta$. Both comparisons with $\gamma$ concern opposite labels and have already been recovered. Every order comparison is therefore determined, and a well-order has no nonidentity order automorphism.
\end{proof}

\begin{lemma}[Generator-pair recovery]\label{lem:fc-pairs}
Suppose an inverse-invariant coloring $C$ recognizes membership in $S$, and its colors recover $p(g)$ for $g\in(S\cup S^2)\setminus\{1\}$. Every permutation preserving the complete $C$-colored graph and fixing $1$ preserves each $B_\alpha$ setwise.
\end{lemma}
\begin{proof}
Such a permutation preserves $S$ and its radial labels $p$. To recover the individual generator inverse pairs among them, for $u,v\in S$ use the test
\[
u\sim v\ \Longleftrightarrow\ u=v\ \text{ or }\ 
\left\{\begin{array}{l}
p(u)=p(v),\\
p(uw^{-1})=p(vw^{-1})\quad\text{for all }w\in S\setminus\{u,v\}.
\end{array}\right.
\]
Every difference tested is nonidentity and has word length at most two, both before and after applying the permutation, so the test is preserved. Its classes are exactly the $B_\alpha$. Elements in the same $B_\alpha$ pass by \eqref{eq:fc-ranks}. If $u\in B_\alpha$, $v\in B_\beta$, and $\alpha<\beta$, different radial labels already separate them. If the radial labels agree, $\gamma=\alpha+1<\beta$ and $w=a_\gamma$ give pair labels $\eta_\gamma\ne\eta_\beta$, separating them instead. On the recovered classes, the vertex and cross-pair labels are those of Lemma~\ref{lem:fc-order}, so that lemma fixes each class setwise.
\end{proof}

\subsection{Finite conflict colorings}
Let $\ell$ be word length for $S$, and let $\mathcal P_n=\{\invp{g}:\ell(g)=n\}$. For each $P\in\mathcal P_n$, choose a representative and a shortest word
\[
g_P=s_{P,1}\cdots s_{P,n},\qquad s_{P,i}\in S,
\]
and set
\[
V(P)=\{s_{P,1}^{\delta_1}\cdots s_{P,n}^{\delta_n}:\delta_i\in\{1,-1\}\}.
\]
There are at most $2^n$ resulting elements. We color the inverse pairs so that no competing sign variation can share a color. Draw an arc $P\to Q$ when $Q\ne P$ and $Q=\invp{v}$ for some $v\in V(P)$ with $\ell(v)=n$. The outdegree is at most $2^n-1$.

A graph obtained by forgetting directions of a directed graph with outdegree at most $d$ has a proper $(2d+1)$-coloring; this is the bounded set-mapping consequence of de Bruijn--Erd\H{o}s \cite[Theorems~1 and~3]{dBE}. For completeness, here is the compactness argument that supplies the coloring. Every finite induced subgraph on $F$ has at most $d|F|$ edges, hence a vertex of degree at most $2d$. Removing such vertices and coloring in reverse order gives a proper $(2d+1)$-coloring of every finite subgraph. In the compact product $\{1,\ldots,2d+1\}^{\mathcal P_n}$, each edge constraint is closed and any finite set of constraints is satisfiable. Compactness then supplies the full coloring.

Choose such colorings
\[
\chi_2:\mathcal P_2\to\{1,\ldots,7\},\qquad
\chi_3:\mathcal P_3\to\{1,\ldots,15\}.
\]
Only distinct inverse pairs of the specified length create edges, so coincident sign variations, variations equal to $1$, and shortened words cause no difficulty.

\subsection{The published short-word input and the finite palette}
For a symmetric set $T$, write
\[
\Xi_{K,T}=\{F\in\Sym(K):F(tx)\in\invp{t}F(x)\quad\text{for all }t\in T,\ x\in K\}.
\]
The short-word theorem of Leemann--de la Salle \cite[Theorem~8]{LdS} gives, for every group and every symmetric generating set,
\begin{equation}\label{eq:fc-shortword}
\Xi_{K,S^{\le3}}=\Xi(K),\qquad S^{\le3}=S\cup S^2\cup S^3.
\end{equation}
The theorem applies without a finite-generation assumption. We only need to translate its convention to the right-Cayley convention used here. Let $J(x)=x^{-1}$. If $F(tx)\in\invp{t}F(x)$, then
\[
(JFJ)(ys)=F(s^{-1}y^{-1})^{-1}\in(JFJ)(y)\invp{s},
\]
which is the cited left-translation convention. Conjugation by $J$ therefore gives \eqref{eq:fc-shortword} in the present convention. This is the only published rigidity input needed below; no uncolored realization is used here.

\begin{proof}[Proof of Theorem~\ref{thm:finitecolor}]
Define the following colors for $g\ne1$:
\[
C(g)=
\begin{cases}
16+p(g),&\ell(g)=1,\\
\chi_2(\invp{g})+7p(g),&\ell(g)=2,\\
\chi_3(\invp{g}),&\ell(g)=3,\\
1,&\ell(g)>3.
\end{cases}
\]
The colors are inverse-invariant. The colors $16$ and $17$ occur precisely on $S$. At length one the color recovers $p$; at length two its value lies in $\{1,\ldots,14\}$ and recovers both $p$ and $\chi_2$. The number of colors is at most
\[
2+\max\{2\cdot7,15\}=17.
\]
Let $F$ preserve every complete difference color. Since the colors $16$ and $17$ mark $S$, $F$ is an automorphism of $\Cay(K,S)$ and therefore preserves its word metric:
\[
\ell(F(x)F(y)^{-1})=\ell(xy^{-1})\qquad\text{for all }x,y\in K.
\] We first need generator-pair control at every base point, not only at the identity. For every base point $x$, the normalized permutation
\[
F_x(z)=F(zx)F(x)^{-1}
\]
also preserves these colors and fixes $1$. The colors recognize $S$ and recover $p$ on $(S\cup S^2)\setminus\{1\}$, so Lemma~\ref{lem:fc-pairs} applies at that base point and gives
\begin{equation}\label{eq:fc-step}
F(sx)\in\invp{s}F(x)\qquad\text{for all }s\in S,\ x\in K.
\end{equation}
The remaining task is to upgrade this generator-step information to all words of length at most three. For a chosen representative $g_P$ of length $n\in\{2,3\}$, iterate \eqref{eq:fc-step} from the rightmost letter to obtain
\[
u=F(g_Px)F(x)^{-1}\in V(P).
\]
Word-metric preservation gives $\ell(u)=\ell(g_P)=n$, and color preservation gives $C(u)=C(g_P)$. At length $n$ the color recovers $\chi_n$, so $\chi_n(\invp{u})=\chi_n(P)$. If $\invp{u}\ne P$, the chosen sign variation creates a conflict edge, contradicting properness of $\chi_n$; hence $\invp{u}=P$.
To obtain the inverse orientation as well, apply the assertion just proved for $g_P$ at the base point $g_P^{-1}x$. This gives
\[
F(x)F(g_P^{-1}x)^{-1}\in\invp{g_P}.
\]
Taking inverses yields
\[
F(g_P^{-1}x)F(x)^{-1}\in\invp{g_P}=\invp{g_P^{-1}},
\]
as required. Together with \eqref{eq:fc-step}, this shows $F\in\Xi_{K,S^{\le3}}$. Equation~\eqref{eq:fc-shortword} then implies $F\in\Xi(K)$. Conversely, every member of $\Xi(K)$ preserves any inverse-invariant coloring of differences.
\end{proof}
\section*{Declaration of generative AI and AI-assisted technologies in the manuscript preparation process}
OpenAI's ChatGPT provided substantial assistance with drafting and revising
this manuscript, locating references, and preparing and checking the submission
files. Its assistance with arguments and computational checks is described in
Section~\ref{sec:checks}. Anthropic's Claude, used through an automated revision pipeline under the author's direction, drafted the restructured presentation of this version (the construction lemma of Section~4, the numbered decoder lemmas of Section~5, and the proofs of Theorem~A and Corollary~6.3) and prepared the source files; its checking of the arguments is described in Section~\ref{sec:checks}. Responsibility for the mathematical content, source
verification, and approval of the final version rests with the author. No AI system is an author of this paper.

\end{document}